\documentclass[a4paper]{amsart}
\usepackage{graphicx,amsmath,amsfonts,latexsym,amssymb,amsthm,mathrsfs, color}
\usepackage[latin1]{inputenc}
\usepackage{yfonts}
\newtheorem{thm}{Theorem}[section]
\newtheorem{proposition}[thm]{Proposition}
\newtheorem{corollary}[thm]{Corollary}
\newtheorem{definition}[thm]{Definition}
\newtheorem{theorem}[thm]{Theorem}
\newtheorem{lemma}[thm]{Lemma}
\newtheorem{remark}[thm]{Remark}

\newtheorem{example}[thm]{Example}

\begin{document}
\title[The Swift-Hohenberg equation on conic manifolds]
{The Swift-Hohenberg equation on conic manifolds}
\author{Nikolaos Roidos}
\address{Department of Mathematics, University of Patras, 26504 Rio Patras, Greece}
\email{roidos@math.upatras.gr}

\begin{abstract}
We consider the Swift-Hohenberg equation on manifolds with conical singularities and show existence, uniqueness and maximal regularity of the short time solution in terms of Mellin-Sobolev spaces. Moreover, we give a necessary and sufficient condition so that the above solution exists for all times. Space asymptotic expansion of the solution near the singularity is also provided and its relation to the local geometry is shown. The same problem is considered on closed manifolds and similar results are obtained by using the above singular analysis theory.
\end{abstract}

\subjclass[2010]{35K65; 35K90; 35K91; 35R01}
\date{\today}
\maketitle

\section{Introduction}

The Swift-Hohenberg equation is the parabolic evolution equation
\begin{eqnarray}\label{CH1}
u'(t)+(\Delta+1)^{2}u(t)&=&F(t,u(t)), \quad t\in(0,T),\\ \label{CH2}
u(0)&=&u_{0},
\end{eqnarray}
where $\Delta$ is the Laplace operator in a given domain, $T>0$ is finite and $F(s,x)$ is a polynomial in $x$ with $s$-dependent complex valued coefficients that are locally Lipschitz continuous in $s\in\mathbb{R}$. Here $u$ is regarded as a scalar field. \eqref{CH1}-\eqref{CH2} has several physical applications; it can model e.g. thermally convecting fluid flows \cite{SH}, cellular flows \cite{PM}, phenomena in optical physics \cite{TGM} etc. It is also well known for its pattern formation under evolution, see, e.g., \cite{CH}, \cite{Mi}, \cite{PT} and \cite{Ue}. 

We consider \eqref{CH1}-\eqref{CH2} on manifolds with conical singularities and employ the theory of cone differential operators together with maximal $L^{q}$-regularity theory for parabolic partial differential equations (PDE). As a first step, in Theorem \ref{RsecD} we show $R$-sectoriality for the Laplacian on manifolds with possibly warped conical tips. Instead of requiring uniform seminorm boundedness of certain parameter dependent families of cone pseudodifferential operators as in \cite[Theorem 5.6]{RS0}, we follow here a direct and independent proof based on a freezing-of-coefficients method. This result then implies maximal $L^{q}$-regularity for the bi-Laplacian in Lemma \ref{RD2}. Then, in Theorem \ref{Tshort}, by applying a theorem of Cl\'ement and Li we show existence, uniqueness and maximal $L^{q}$-regularity of a short time solution $u$ in terms of Mellin-Sobolev spaces. Maximal continuous regularity for $u$ is also provided.

Next, in Theorem \ref{longt}, we obtain the long time existence of the above solution, by imposing an additional assumption on the right hand side of \eqref{CH1}, which turns out to be necessary for this property. Namely, we show that if the $L^{q}$-norm on $(0,T)$ of the term $F(\cdot,u(\cdot))$ does not blow up in finite $T$, then the Banach fixed point argument of the Cl\'ement and Li theorem can be repeated successively and hence the well-posed solution exists for all times. The crucial step here is the interpolation space estimate of the solution, see \eqref{intbound}, which is obtained by using a maximal $L^{q}$-regularity inequality for linear parabolic problems.

As a consequence, information concerning the asymptotic behavior of the above solution close to the singularity is also provided and its relation to the local geometry is established. These space asymptotics are obtained either through the domain of the bi-Laplacian or through the real interpolation space between the bi-Laplacian domain and the underlying space, see \eqref{DD22} and Proposition \ref{lpo}. In both cases the solution is expressed as a sum of two terms. The first one is a sum of complex powers and logarithmic integer powers of the geodesic distance to the singularity; these powers are determined explicitly by the local geometry around the conical tip, see \eqref{DD22}. The second one is a remainder that belongs to a Mellin-Sobolev space and close to the singularity it decays to zero faster than each of the components of the previous term; its decay is determined by the local geometry as well as by the initial data.

Finally, we consider \eqref{CH1}-\eqref{CH2} on closed manifolds. Instead of using standard techniques, we employ geodesic polar coordinates and study the resulting degenerate problem by applying the above singular analysis theory. 

\section{Maximal $L^{q}$-regularity theory for linear and quasilinear parabolic problems}

In this section we present some basic functional analytic machinery related to the property of maximal $L^q$-regularity for linear and quasilinear parabolic problems. Let $X_1\overset{d}{\hookrightarrow} X_0$ be a continuously and densely injected complex Banach couple.

\begin{definition}[Sectorial operators]\label{sectorial}
Let $\mathcal{P}(K,\theta)$, $K\geq1$, $\theta\in[0,\pi)$, be the class of all closed densely defined linear operators $A$ in $X_{0}$ such that $S_{\theta}=\{\lambda\in\mathbb{C}\,|\, |\arg(\lambda)|\leq\theta\}\cup\{0\}\subset\rho{(-A)}$ and $(1+|\lambda|)\|(A+\lambda)^{-1}\|_{\mathcal{L}(X_{0})}\leq K$ when $\lambda\in S_{\theta}$. The elements in $\mathcal{P}(\theta)=\cup_{K\geq1}\mathcal{P}(K,\theta)$ are called {\em (invertible) sectorial operators of angle $\theta$}. If $A\in\mathcal{P}(\theta)$ then any $K\geq1$ such that $A\in \mathcal{P}(K,\theta)$ is called {\em sectorial bound of $A$}.
\end{definition}

If in the definition of sectoriality we replace the standard boundedness by the Rademacher boundedness, then we obtain the following property.

\begin{definition}[$R$-sectorial operators]
Denote by $\mathcal{R}(K,\theta)$, $K\geq1$, $\theta\in[0,\pi)$, the class of all operators $A\in \mathcal{P}(\theta)$ in $X_{0}$ such that for any choice of $\lambda_{1},...,\lambda_{N}\in S_{\theta}\backslash\{0\}$ and $x_{1},...,x_{N}\in X_0$, $N\in\mathbb{N}\backslash\{0\}$, we have
$$
\|\sum_{k=1}^{N}\epsilon_{k}\lambda_{k}(A+\lambda_{k})^{-1}x_{k}\|_{L^{2}(0,1;X_0)} \leq K \|\sum_{k=1}^{N}\epsilon_{k}x_{k}\|_{L^{2}(0,1;X_0)},
$$
where $\{\epsilon_{k}\}_{k=1}^{\infty}$ is the sequence of the Rademacher functions. The elements in $\mathcal{R}(\theta)=\cup_{K\geq1}\mathcal{R}(K,\theta)$ are called {\em $R$-sectorial operators of angle $\theta$}. If $A\in \mathcal{R}(\theta)$ then any $K\geq1$ such that $A\in \mathcal{R}(K,\theta)$ is called {\em $R$-bound of $A$}.
\end{definition} 

Sectorial operators admit a holomorphic functional calculus defined by the Dunford integral formula, see, e.g., \cite[Chapter 3.1.4]{PrS}. For any $\rho\geq0$ and $\theta\in(0,\pi)$ denote by $\Gamma_{\rho,\theta}$ the boundary of $S_{\theta}\cup \{\lambda\in\mathbb{C}\, |\, |\lambda|\leq\rho\}$ oriented counterclockwise. Moreover, we denote $\Gamma_{0,\theta}$ simply by $\Gamma_{\theta}$. Then, given any $A\in \mathcal{P}(\theta)$ we can define its complex powers $A^z$ for $\mathrm{Re}(z)<0$ by
$$
A^{z}=\frac{1}{2\pi i}\int_{\Gamma_{\rho,\theta}}(-\lambda)^{z}(A+\lambda)^{-1}d\lambda \in \mathcal{L}(X_{0}),
$$
for certain $\rho>0$. The above family together with $A^{0} = I$ is a strongly continuous holomorphic semigroup on $X_{0}$, see, e.g., \cite[Theorem III.4.6.2]{Am} and \cite[Theorem III.4.6.5]{Am}. The definition can be extended to any $z\in\mathbb{C}$, and for $\mathrm{Re}(z)\geq0$, $A^{z}$ are in general unbounded operators, see, e.g., \cite[Theorem III.4.6.5]{Am}. Of particular interest from the PDE point of view is the case when the purely imaginary powers turn out to be bounded operators. 

\begin{definition}[Bounded imaginary powers]
Let $A\in\mathcal{P}(\theta)$, $\theta\in[0,\pi)$, in $X_{0}$. We say that $A$ has {\em bounded imaginary powers} if there exists some $\varepsilon>0$ and $\delta\geq1$ such that $A^{it}\in\mathcal{L}(X_{0})$ and $\|A^{it}\|_{\mathcal{L}(X_{0})}\leq \delta$ for all $t\in[-\varepsilon,\varepsilon]$. In this case, {\em(}see, e.g., \cite[Corollary III.4.7.2]{Am}{\em)} $A^{it}\in\mathcal{L}(X_{0})$ for all $t\in\mathbb{R}$ and there exists a $\phi\geq0$, called {\em power angle of $A$}, such that $\|A^{it}\|_{\mathcal{L}(X_{0})}\leq M e^{\phi|t|}$, $t\in\mathbb{R}$, with some $M\geq1$. We write $A\in\mathcal{BIP}(\phi)$.
\end{definition} 

Next we focus on parabolic PDE. Denote by $L^{q}(0,T;X_{0})$ and $W^{1,q}(0,T;X_{0})$, $q\in(1,\infty)$, $T>0$, the space of the $X_{0}$-valued $L^{q}$ and $W^{1,q}$ functions respectively. By starting with the linear theory we recall the following property.

\begin{definition}[Maximal $L^q$-regularity]
Let $q\in(1,\infty)$, $T>0$, $-A:\mathcal{D}(A)=X_{1}\rightarrow X_{0}$ be the infinitesimal generator of an analytic semigroup on $X_{0}$ and consider the Cauchy problem
\begin{eqnarray}\label{AP}
u'(t)+Au(t)&=&g(t), \quad t\in(0,T),\\\label{AP2}
u(0)&=&0.
\end{eqnarray}
The operator $A$ has {\em maximal $L^{q}$-regularity} if for any $g\in L^{q}(0,T;X_{0})$ there exists a unique $u\in W^{1,q}(0,T;X_{0})\cap L^{q}(0,T;X_{1})$ solving \eqref{AP}-\eqref{AP2}, that depends continuously on $g$. 
\end{definition}

Note that the above property is $T$-independent and according to a result of G. Dore, also $q$-independent. Moreover, by setting $u(t)=e^{ct}v(t)$, $c>0$, we can insert a positive shift $c$ to the operator $A$ in \eqref{AP}. Finally, we recall the standard embedding of the maximal $L^{q}$-regularity space, namely
\begin{gather}\label{embmaxreg}
W^{1,q}(\tau_{1},\tau_{2};X_{0})\cap L^{q}(\tau_{1},\tau_{2};X_{1})\hookrightarrow C([\tau_{1},\tau_{2}];(X_{1},X_{0})_{\frac{1}{q},q}),\quad 0\leq\tau_{1}<\tau_{2}\leq T<\infty,
\end{gather}
see, e.g., \cite[Theorem III.4.10.2]{Am}; note that if we restrict to the subspace of $W^{1,q}(\tau_{1},\tau_{2};X_{0})\cap L^{q}(\tau_{1},\tau_{2};X_{1})$ consisting of functions $u$ that satisfy $u(0)=0$, then the norm of the above embedding is independent of $\tau_{1},\tau_{2}\in[0,T]$, see, e.g., \cite[Corollary 2.3]{CL}. Here $(X_{0},X_{1})_{\xi,q}$, $\xi\in(0,1)$, denotes real interpolation. 

All the spaces we consider in the sequel belong to the class of UMD (unconditionality of martingale differences property, see, e.g., \cite[Section III.4.4]{Am}) spaces. By using the underlying geometric properties of such spaces, we can relate the boundedness of the imaginary powers of an operator to the maximal $L^q$-regularity by the following classical result.

\begin{theorem}{\rm (Dore and Venni, \cite[Theorem 3.2]{DV}).}\label{rem1}
If $X_{0}$ is UMD and $A\in \mathcal{BIP}(\phi)$ in $X_{0}$ with $\phi<\frac{\pi}{2}$, then $A$ has maximal $L^q$-regularity.
\end{theorem}

However, in UMD spaces the $R$-sectoriality is weaker than the boundedness of the imaginary powers property (see \cite[Theorem 4]{CP}) and it turns out to be sufficient for maximal $L^q$-regularity; it actually leads to characterization of this property in UMD spaces due to \cite[Theorem 4.2]{W}.

\begin{theorem}{\rm (Kalton and Weis, \cite[Theorem 6.5]{KW1}).}\label{KaW}
If $X_{0}$ is UMD and $A\in \mathcal{R}(\theta)$ in $X_{0}$ with $\theta>\frac{\pi}{2}$, then $A$ has maximal $L^q$-regularity.
\end{theorem}

Next we state the well known version of Theorem \ref{KaW} for non-zero initial data. In comparison to the standard version, Proposition \ref{propmax}, whose proof can be seen in the Appendix, contains better information on the constants.

\begin{proposition}\label{propmax}
Let $A:\mathcal{D}(A)=X_{1}\rightarrow X_{0}$ such that $A\in \mathcal{P}(K,\theta)\cap \mathcal{R}(R,\theta)$, for certain $K,R\geq1$ and $\theta\in(\frac{\pi}{2},\pi)$. Assume that $X_{0}$ is UMD and let $g\in \cap_{\tau>0}L^{q}(0,\tau;X_{0})$ and $u_{0}\in (X_{1},X_{0})_{\frac{1}{q},{q}}$, for some $q\in(1,\infty)$. Then, for any $T>0$ the Cauchy problem 
\begin{eqnarray}\label{APF}
u'(t)+Au(t)&=&g(t), \quad t\in(0,T),\\\label{thr}
u(0)&=&u_{0} 
\end{eqnarray}
has a unique solution $u\in W^{1,q}(0,T;X_{0})\cap L^{q}(0,T;X_{1})$ and the following estimate holds
\begin{gather}\label{maxreg}
\|u\|_{W^{1,q}(0,T;X_{0})}+\|u\|_{L^{q}(0,T;X_{1})}\leq C(\|g\|_{L^{q}(0,T;X_{0})}+\|u_{0}\|_{(X_{1},X_{0})_{\frac{1}{q},{q}}}),
\end{gather}
with some constant $C>0$ that only depends on $A$, $R$, $\theta$, and $q$. 

If in addition $g\in \cap_{\tau>0}L^{\infty}(0,\tau;(X_{0},X_{1})_{\phi,\eta})$ for some $\phi\in(0,1)$ and $\eta\in(1,\infty)$, then $u(t)\in X_{1}$ for each $t>0$, and for any $\alpha\in(0,\phi)$ and $\beta\in(-1,-\frac{1}{q})$ the following estimate holds
\begin{gather}\label{x1bound}
\|u(t)\|_{X_{1}}\leq Mt^{\alpha}\|g\|_{L^{\infty}(0,t;(X_{0},X_{1})_{\phi,\eta})}+L t^{\beta}\|u_{0}\|_{(X_{1},X_{0})_{\frac{1}{q},q}}, \quad t>0,
\end{gather}
with some constants $M>0$ and $L>0$ that only depend on $A$, $K$, $\theta$, $\phi$, $\eta$, $\alpha$ and $A$, $K$, $\theta$, $q$, $\beta$ respectively. Moreover, if $g\in C([0,\infty);(X_{0},X_{1})_{\phi,\eta})$, then 
\begin{gather}\label{contregabs}
u\in C^{1}((0,\infty);X_{0}) \cap C((0,\infty);X_{1}).
\end{gather}
\end{proposition}

\begin{corollary}\label{cormax}
Let the assumptions of Proposition \ref{propmax} be satisfied with $A$ replaced by $A+c$, for a certain $c\geq0$. Then, for any $T>0$ the Cauchy problem \eqref{APF}-\eqref{thr} has a unique solution $u\in W^{1,q}(0,T;X_{0})\cap L^{q}(0,T;X_{1})$ and the following estimate holds
\begin{gather}\label{maxreg2}
\|u\|_{W^{1,q}(0,T;X_{0})}+\|u\|_{L^{q}(0,T;X_{1})}\leq C e^{cT}(\|g\|_{L^{q}(0,T;X_{0})}+\|u_{0}\|_{(X_{1},X_{0})_{\frac{1}{q},{q}}}),
\end{gather}
with some constant $C>0$ that only depends on $A$, $R$, $\theta$, $q$, and $c$. 
\end{corollary}
\begin{proof}
Follows by setting $u(t)=e^{ct}f(t)$ in \eqref{APF}-\eqref{thr} and then applying Proposition \ref{propmax} and \eqref{maxreg} to the resulting problem.
\end{proof}

We regard \eqref{CH1}-\eqref{CH2} as a quasilinear parabolic equation. Hence, we consider problems of the form 
\begin{eqnarray}\label{QL}
u'(t)+A(u(t))u(t)&=&G(t,u(t))+g(t),\quad t\in(0,T),\\\label{QL2}
u(0)&=&u_{0}.
\end{eqnarray}
We end this section with the following well known short time maximal $L^{q}$-regularity result.

\begin{theorem}\label{CL} {\rm (Cl\'ement and Li, \cite[Theorem 2.1]{CL}).}
Let $u_{0}\in (X_{1},X_{0})_{\frac{1}{q},{q}}$ for some $q\in(1,\infty)$ and let $U\subset (X_{1},X_{0})_{\frac{1}{q},{q}}$ be an open set such that $u_{0}\in U$. Assume that{\em:}\\
{\em (H1)} $A(u_0): X_{1}\rightarrow X_{0}$ has maximal $L^{q}$-regularity.\\
{\em (H2)} $A(\cdot)\in C^{1-}(U; \mathcal{L}(X_1,X_0))$.\\
{\em (H3)} $G(\cdot,\cdot)\in C^{1-,1-}([0,T_0]\times U; X_0)$ for certain $T_{0}>0$.\\
{\em (H4)} $g\in L^q(0,T_0; X_0)$.\\
Then, there exists a $T\in(0,T_{0}]$ and a unique $u\in W^{1,q}(0,T;X_{0})\cap L^{q}(0,T;X_{1})$ solving the equation \eqref{QL}-\eqref{QL2} on $[0,T)$.
\end{theorem}

\section{The Laplacian on manifolds with conical singularities}

Let $\mathcal{B}$ be an $(n+1)$-dimensional, $n\geq1$, smooth compact connected manifold with possibly disconnected boundary $\partial \mathcal{B}$. We endow $\mathcal{B}$ with a Riemannian metric $\mathfrak{g}$ such that when it is restricted to the collar neighborhood $[0,1)\times \partial \mathcal{B}$ of the boundary, in local coordinates it admits the warped product form 
\begin{gather}\label{metric}
\mathfrak{g}|_{[0,1)\times \partial \mathcal{B}}=dx^{2}+x^{2}\mathfrak{h}(x),
\end{gather}
where $x\in [0,1)$ and the map $x\mapsto \mathfrak{h}(x)$ is a smooth up to $x=0$ family of Riemannian metrics on the cross section $\partial \mathcal{B}$. We call $\mathbb{B}=(\mathcal{B},\mathfrak{g})$ {\em conic manifold} or {\em manifold with conical singularities}, which are identified with the subset $\{0\}\times\partial \mathcal{B}$ of $\mathcal{B}$. If $\mathfrak{h}$ is independent of $x$ when $x$ is close to zero, then we have straight conical tips. Denote $\partial\mathbb{B}=(\partial \mathcal{B},\mathfrak{h}(0))$ and $\partial\mathcal{B}=\sqcup_{i=1}^{k_{\mathcal{B}}}\partial\mathcal{B}_{i}$, for certain $k_{\mathcal{B}}\in\mathbb{N}\backslash\{0\}$, where $\partial\mathcal{B}_{i}$ are smooth, closed and connected.

The naturally appearing differential operators on a conic manifold belong to the class of {\em cone differential operators} or {\em Fuchs type operators}. An $\mu$-th order cone differential operator $A$, $\mu\in\mathbb{N}$, is an $\mu$-th order differential operator with smooth coefficients in the interior $\mathcal{B}^{\circ}$ of $\mathcal{B}$ such that near the boundary, e.g., on the collar neighborhood $(0,1)\times \partial\mathcal{B}$, it admits the following form $A=x^{-\mu}\sum_{k=0}^{\mu}a_{k}(x)(x\partial_{x})^{k}$, where $a_{k}\in C^{\infty}([0,1);\mathrm{Diff}^{\mu-k}(\partial\mathbb{B}))$. For details concerning the associated pseudodifferential theory, that is called {\em cone calculus}, we refer to \cite{GM}, \cite{GKM}, \cite{Le}, \cite{RS0}, \cite{RS2}, \cite{RS3}, \cite{SS1}, \cite{SS2}, \cite{Sh}, \cite{Schu} and \cite{Sei}. In the sequel we recall some basic facts focussing more on the theory of the Laplacian. 

Cone differential operators preserve the scales of weighted {\em Mellin-Sobolev spaces} $\mathcal{H}_{p}^{s,\gamma}(\mathbb{B})$, $p\in(1,\infty)$, $s,\gamma\in\mathbb{R}$, i.e. such an operator $A$ of order $\mu$ induces a bounded map from $\mathcal{H}^{s+\mu,\gamma+\mu}_p(\mathbb{B})$ to $\mathcal{H}^{s,\gamma}_p(\mathbb{B})$. Let $\omega$ be a fixed cut-off function near $\{0\}\times\partial \mathcal{B}$, i.e. a smooth non-negative function $\omega$ on $\mathcal{B}$ such that $\omega=1$ near $\{0\}\times\partial \mathcal{B}$ and $\omega=0$ outside $[0,1)\times \partial \mathcal{B}$; for simplicity we choose $\omega$ to be only $x$-dependent. If we denote by $C_{c}^{\infty}$ the space of smooth compactly supported functions, then we have the following.

\begin{definition}[Mellin-Sobolev spaces]
For any $\gamma\in\mathbb{R}$ consider the map $M_{\gamma}: C_{c}^{\infty}(\mathbb{R}_{+}\times\mathbb{R}^{n})\rightarrow C_{c}^{\infty}(\mathbb{R}^{n+1})$ defined by $u(x,y)\mapsto e^{(\gamma-\frac{n+1}{2})x}u(e^{-x},y)$. Furthermore, take a covering $\kappa_{i}:U_{i}\subseteq\partial\mathcal{B} \rightarrow\mathbb{R}^{n}$, $i\in\{1,...,N\}$, $N\in\mathbb{N}\backslash\{0\}$, of $\partial\mathcal{B}$ by coordinate charts and let $\{\phi_{i}\}_{i\in\{1,...,N\}}$ be a subordinated partition of unity. For any $p\in(1,\infty)$ and $s\in\mathbb{R}$ let $\mathcal{H}^{s,\gamma}_p(\mathbb{B})$ be the space of all distributions $u$ on $\mathbb{B}^{\circ}$ such that 
$$
\|u\|_{\mathcal{H}^{s,\gamma}_p(\mathbb{B})}=\sum_{i=1}^{N}\|M_{\gamma}(1\otimes \kappa_{i})_{\ast}(\omega\phi_{i} u)\|_{H^{s}_p(\mathbb{R}^{n+1})}+\|(1-\omega)u\|_{H^{s}_p(\mathbb{B})}
$$
is defined and finite, where $\ast$ refers to the push-forward of distributions. 
\end{definition}

Up to norm equivalence, $\mathcal{H}_{p}^{s,\gamma}(\mathbb{B})$ is independent of the choice of the covering $\{\kappa_{i}\}_{i\in\{1,...,N\}}$, the partition $\{\phi_{i}\}_{i\in\{1,...,N\}}$ and the cut-off function $
\omega$. Moreover, since the usual Sobolev spaces are UMD, by \cite[Theorem III.4.5.2]{Am}, the Mellin-Sobolev spaces are also UMD. In addition, if $s\in \mathbb{N}$, then $\mathcal{H}^{s,\gamma}_p(\mathbb{B})$ is the space of all functions $u$ in $H^s_{p,loc}(\mathbb{B}^\circ)$ such that, near the boundary
$$
x^{\frac{n+1}2-\gamma}(x\partial_x)^{k}\partial_y^{\alpha}(\omega(x) u(x,y)) \in L^{p}([0,1)\times \partial \mathcal{B}, \sqrt{\mathrm{det}[\mathfrak{h}(x)]}\frac{dx}xdy),\quad k+|\alpha|\leq s.
$$

We regard a cone differential operator $A$ of order $\mu\geq1$ as an unbounded operator in $\mathcal{H}^{s,\gamma}_p(\mathbb{B})$ with domain $C_{c}^{\infty}(\mathbb{B}^{\circ})$. If $A$ is $\mathbb{B}$-elliptic, see \cite[Definition 2.1]{Sh}, then the domain of its minimal extension (i.e. its closure) $\underline{A}_{s,\min}$ differs from the domain of the maximal extension $\underline{A}_{s,\max}$ of $A$, which is defined as usual by $\mathcal{D}(\underline{A}_{s,\max})=\{u\in\mathcal{H}^{s,\gamma}_p(\mathbb{B}) \, |\, Au\in \mathcal{H}^{s,\gamma}_p(\mathbb{B})\}$, by an $s$-independent finite-dimensional space $\mathcal{E}_{A,\gamma}$. The space $\mathcal{E}_{A,\gamma}$, which is called {\em asymptotics space}, consists of smooth functions on $\mathbb{B}^{\circ}$ that vanish on $\mathcal{B}\backslash([0,1)\times\partial\mathcal{B})$ and in local coordinates on $(0,1)\times\partial\mathcal{B}$ they are linear combinations of functions of the form $\omega(x)c(y)x^{-\rho}\log^{\kappa}(x)$, where $c\in C^{\infty}(\partial\mathbb{B})$, $\rho\in\mathbb{C}$ and $\kappa\in\mathbb{N}$. The exponents $\rho$ lie in the strip $I_{\gamma,\mu}=\{\lambda\in\mathbb{C}\, |\, \mathrm{Re}(\lambda)\in [\frac{n+1}{2}-\gamma-\mu,\frac{n+1}{2}-\gamma)\}$ and together with the exponents $\kappa$ they are determined explicitly by the coefficients of $A$. Moreover, there is a one to one correspondence between the possible closed extensions of $A$, called also {\em realizations}, and the subspaces of $\mathcal{E}_{A,\gamma}$.

The Laplacian $\Delta$ induced by the metric \eqref{metric} on the collar part $(0,1)\times \partial\mathcal{B}$ has the following form
$$
\Delta=\frac{1}{x^{2}}\Big((x\partial_x)^{2}+(n-1+\frac{x\partial_x (\det[\mathfrak{h}(x)])}{2\det[\mathfrak{h}(x)]})(x\partial_x )+\Delta_{\mathfrak{h}(x)} \Big),
$$
where $\Delta_{\mathfrak{h}(x)}$ is the Laplacian on $\partial\mathcal{B}$ induced by the metric $\mathfrak{h}(x)$. $\Delta$ is a second order $\mathbb{B}$-elliptic cone differential operator. 

Denote by $\mathbb{C}_{\omega}$ the space of all $C^{\infty}(\mathbb{B}^\circ)$ functions $c$ that vanish on $\mathcal{B}\backslash([0,1)\times\partial\mathcal{B})$ and on each component $[0,1)\times\partial\mathcal{B}_{i}$, $i\in\{1,...,k_{\mathbb{B}}\}$, they are of the form $c_{i}\omega$, where $c_{i}\in\mathbb{C}$, i.e. $\mathbb{C}_{\omega}$ consists of smooth functions that are locally constant close to the boundary. Endow $\mathbb{C}_{\omega}$ with the norm given by $c\mapsto (\sum_{i=1}^{k_{\mathcal{B}}}|c_{i}|^{2})^{\frac{1}{2}}$. If $\gamma\in(\frac{n-3}{2},\frac{n+1}{2})$ then zero is a pole of the inverse of the {\em conormal symbol} of $\Delta$, i.e. of the map $\mathbb{C}\ni \lambda\mapsto \lambda^{2}-(n-1)\lambda+\Delta_{\mathfrak{h}(0)} \in \mathcal{L}(H_{2}^{2}(\partial\mathbb{B}),H_{2}^{0}(\partial\mathbb{B}))$, see, e.g., \cite[Section 2.1]{Sh} for definition, that lies in the strip $I_{\gamma,2}$. In this situation $\mathbb{C}_{\omega}$ is a subspace of $\mathcal{E}_{\Delta,\gamma}$. Such a realization can satisfy the property of maximal $L^q$-regularity as we can see in the following result.

\begin{theorem}\label{RsecD}
Let $p\in(1,\infty)$, $s\geq0$ and 
\begin{eqnarray}\label{choicegg}
\frac{n-3}2<\gamma<\min\Big\{-1+\sqrt{\Big(\frac{n-1}{2}\Big)^{2}-\lambda_{1}} ,\frac{n+1}{2}\Big\},
\end{eqnarray}
where $\lambda_{1}$ is the greatest non-zero eigenvalue of the boundary Laplacian $\Delta_{\mathfrak{h}(0)}$. Consider the closed extension $\underline{\Delta}_{s}$ of $\Delta$ in $X_{0}^{s}=\mathcal{H}_{p}^{s,\gamma}(\mathbb{B})$ with domain 
\begin{gather}\label{DD}
\mathcal{D}(\underline{\Delta}_{s})=X_{1}^{s}=\mathcal{H}_{p}^{s+2,\gamma+2}(\mathbb{B})\oplus\mathbb{C}_{\omega}.
\end{gather}
Then, for any $\theta\in[0,\pi)$ there exists a $c>0$ such that $c-\underline{\Delta}_{s}\in\mathcal{R}(\theta)$.
\end{theorem}
\begin{proof}
The result can be recovered by \cite[Theorem 5.6]{RS0}. However, we outline here an alternative proof inspired by the proof of \cite[Theorem 6.1]{RS2} which we follow.

Recall that $\mathbb{B}=(\mathcal{B},\mathfrak{g})$ and consider another Riemannian metric $\mathfrak{g}_{0}$ on $\mathcal{B}$ such that $\mathfrak{g}_{0}|_{[0,\xi)\times\partial \mathcal{B}}=dx^{2}+x^{2}\mathfrak{h}(0)$ and $\mathfrak{g}_{0}=\mathfrak{g}$ on $\mathcal{B}\backslash([0,1)\times\partial \mathcal{B})$, for some fixed $\xi\in(0,1)$. Denote $\mathbb{B}_{0}=(\mathcal{B},\mathfrak{g}_{0})$ and let $\Delta_{\mathfrak{g}_{0}}$ be the associated Laplacian. Note that $\Delta$ and $\Delta_{\mathfrak{g}_{0}}$ have the same conormal symbol. Let $\underline{\Delta}_{\mathfrak{g}_{0},s}$ be the realization of $\Delta_{\mathfrak{g}_{0}}$ in $\mathcal{H}^{s,\gamma}_{p}(\mathbb{B}_{0})$ with domain $\mathcal{H}^{s+2,\gamma+2}_{p}(\mathbb{B}_{0})\oplus \mathbb{C}_{\omega}$. 

Assume first that $s=0$. Let $\underline{A}_{0}$ be the realization of $\Delta_{\mathfrak{g}_{0}}$ in $X_{0}^{0}$ with domain $X_{1}^{0}$. By the choice of the domain of $\underline{\Delta}_{\mathfrak{g}_{0},0}$, from \cite[Theorem 2.9]{RS3}, \cite[Remark 2.10]{RS3}, \cite[Theorem 5.6]{Sh} and \cite[Theorem 5.7]{Sh} we have that $\underline{\Delta}_{\mathfrak{g}_{0},0}$ satisfies the ellipticity condition (E1), (E2) and (E3) of \cite[Section 3.2]{Sh}. Hence, since $\underline{\Delta}_{\mathfrak{g}_{0},0}$ and $\underline{A}_{0}$ have the same model cone operator (see \cite[Section 2.4]{Sh} for definition), we deduce that (E1), (E2) and (E3) are also fulfilled by $\underline{A}_{0}$. Therefore, by \cite[Theorem 4.3]{Sh}, for any $\phi\in(0,\pi)$ there exists some $c>0$ such that $c-\underline{A}_{0}\in \mathcal{BIP}(\phi)$. As a consequence, since $X_{0}^{0}$ is UMD, by \cite[Theorem 4]{CP} there exists some $c>0$ such that $c-\underline{A}_{0}\in\mathcal{R}(\theta)$. 

Take a cut-off function $\widetilde{\omega}$ on $\mathcal{B}$ with values in $[0,1]$ satisfying $\widetilde{\omega}=1$ on $ [0,\sigma]\times\partial\mathcal{B}$ and $\widetilde{\omega}=0$ on $\mathcal{B}\backslash ([0,2\sigma]\times\partial\mathcal{B})$, where $\sigma\in(0,\frac{1}{2})$. Consider the operator $\underline{A}_{1}=\widetilde{\omega}\underline{\Delta}_{0}+(1-\widetilde{\omega})\underline{A}_{0}$ with domain $X_{1}^{0}$ in $X_{0}^{0}$. By taking $\sigma$ sufficiently small, we can make the $\mathcal{L}(X_{1}^{0},X_{0}^{0})$-norm of $\widetilde{\omega}(\underline{\Delta}_{0}-\underline{A}_{0})$ arbitrary small. Hence, by the relation $\underline{A}_{1}=\underline{A}_{0}+\widetilde{\omega}(\underline{\Delta}_{0}-\underline{A}_{0})$ and the perturbation result \cite[Theorem 1]{KL}, there exists some $\sigma\in(0,\frac{1}{2})$ such that $c-\underline{A}_{1}\in\mathcal{R}(\theta)$. 

Endow $\mathcal{B}$ with a metric $\mathfrak{g}_{2}$ such that $\mathfrak{g}_{2}|_{[0,\eta)\times\partial \mathcal{B}}=dx^{2}+x^{2}\mathfrak{h}(0)$ and $\mathfrak{g}_{2}=\mathfrak{g}$ on $\mathcal{B}\backslash([0,2\eta]\times\partial \mathcal{B})$, where $\eta\in(0,\frac{\sigma}{2})$. Denote by $\Delta_{\mathfrak{g}_{2}}$ the Laplacian associated with $(\mathcal{B},\mathfrak{g}_{2})$ and by $\underline{A}_{2}$ its realization in $X_{0}^{0}$ with domain $X_{1}^{0}$. Similarly to $\underline{A}_{0}$, there exists some $c>0$ such that $c-\underline{A}_{2}\in\mathcal{R}(\theta)$. 

Take a smooth function $\phi_{1}$ on $\mathcal{B}$ such that $\mathrm{supp}(\phi_{1})\subset [0,\sigma)\times\partial\mathcal{B}$ and $\phi_{1}=1$ in $[0,2\eta+\delta)\times\partial\mathcal{B}$, for some $\delta\in (0,\sigma-2\eta)$. Assume that $\phi_{1}$ depends only on the $x$ variable and let $\phi_{2}=1-\phi_{1}$. Furthermore, let $\psi_{1}$, $\psi_{2}$ be two smooth functions on $\mathcal{B}$ such that $\mathrm{supp}(\psi_{1})\subset [0,\sigma)\times\partial\mathcal{B}$, $\mathrm{supp}(\psi_{2})\subset \mathcal{B}\backslash([0,2\eta)\times\partial\mathcal{B})$, $\psi_{1}=1$ on $\mathrm{supp}(\phi_{1})$ and $\psi_{2}=1$ on $\mathrm{supp}(\phi_{2})$. Then, we can proceed similarly to \cite[(6.5)]{RS2} and \cite[(6.7)]{RS2} and construct respectively a left and right inverse of $\lambda-\underline{\Delta}_{0}$ for large $|\lambda|$ with $\lambda\in S_{\theta}$. For such $\lambda$ this construction provides the following resolvent representation, namely $(\lambda-\underline{\Delta}_{0})^{-1}=\sum_{k=0}^{\infty}(-1)^{k}Q^{k}(\lambda)R(\lambda)$, where $R(\lambda)=\sum_{i=1}^{2}\psi_{i}(\lambda-\underline{A}_{i})^{-1}\phi_{i}$ and $Q(\lambda)=\sum_{i=1}^{2}\psi_{i}(\lambda-\underline{A}_{i})^{-1}[\underline{\Delta}_{0},\phi_{i}]$, see also \cite[(6.8)]{RS2}. From the above representation, $R$-sectoriality for $c-\underline{\Delta}_{0}$ with $c>0$ sufficiently large can be shown as in \cite[(6.9)]{RS2}-\cite[(6.12)]{RS2}.

Consider now the case $s>0$. By Step 1 in the proof of \cite[Theorem 3.3]{RS0} we have that $\rho(\underline{\Delta}_{0})\subseteq\rho(\underline{\Delta}_{s})$ and that the resolvent of $\underline{\Delta}_{s}$ for $s>0$ is the restriction of the resolvent of $\underline{\Delta}_{0}$ to $X_{0}^{s}$. Then, we can first treat the case of $s\in\mathbb{N}$ by induction as in \cite[(6.13)]{RS2}. Finally, the case when $s>0$ is not an integer follows by interpolation by \cite[Lemma 3.7]{RS2}, \cite[Theorem 3.19]{KS} and \cite[Lemma 2.6]{RS2}.
\end{proof}

Next we focus on the bi-Laplacian $\underline{\Delta}_{s}^{2}$ associated with the Laplacian $\underline{\Delta}_{s}$ from \eqref{DD}. The domain of $\underline{\Delta}_{s}^{2}$, defined as usual by $\mathcal{D}(\underline{\Delta}_{s}^{2})=\{u\in\mathcal{D}(\underline{\Delta}_{s})\, |\, \underline{\Delta}_{s}u\in \mathcal{D}(\underline{\Delta}_{s})\}$, has the following form
\begin{gather}\label{DD22}
\mathcal{D}(\underline{\Delta}_{s}^{2})=\mathcal{D}(\underline{\Delta}_{s,\min}^{2})\oplus\mathbb{C}_{\omega}\oplus\mathcal{F}_{\Delta,\gamma}.
\end{gather}
Here for the minimal domain we have that $\mathcal{D}(\underline{\Delta}_{s,\min}^{2})\subset \mathcal{H}_{p}^{s+4,\gamma+4-\varepsilon}(\mathbb{B})$, for any $\varepsilon>0$. Moreover, the asymptotics space $\mathcal{F}_{\Delta,\gamma}$ is an $s$-independent finite dimensional space that consists of smooth functions on $\mathbb{B}^{\circ}$ that vanish on $\mathcal{B}\backslash([0,1)\times\partial\mathcal{B})$ and in local coordinates on the collar part $(0,1)\times\partial\mathcal{B}$ they are linear combinations of functions of the form $\omega(x)c(y)x^{-\rho}\log^{k}(x)$, $k\in\{0,1,2,3\}$, where $c\in C^{\infty}(\partial\mathbb{B})$ and the exponents $\rho$ lie in the strip $\{\lambda\in\mathbb{C}\, |\, \mathrm{Re}(\lambda)\in [\frac{n-7}{2}-\gamma,\frac{n-3}{2}-\gamma)\}$. In particular we have $\mathcal{F}_{\Delta,\gamma}\subset \mathcal{H}_{p}^{s+2,\gamma+2}(\mathbb{B})$.

We end this section by showing that the above bi-Laplacian satisfies the property of maximal $L^q$-regularity. 
\begin{lemma}\label{RD2}
Let $p\in(1,\infty)$, $s\geq0$, $\gamma$ be chosen as in \eqref{choicegg} and $\underline{\Delta}_{s}$ be the Laplacian \eqref{DD}. Then, for any $\theta\in[0,\pi)$ there exists a $c>0$ such that $\underline{\Delta}_{s}^{2}+c\in\mathcal{R}(\theta)$. 
\end{lemma}
\begin{proof}
By Theorem \ref{RsecD} there exists a $c_{0}$ such that $c_{0}-\underline{\Delta}_{s}\in \mathcal{R}(\frac{\pi+\theta}{2})$. Hence, for any $\lambda_{1},...,\lambda_{N}\in S_{\theta}\backslash\{0\}$ and $x_{1},...,x_{N}\in X_{0}^{s}$, $N\in\mathbb{N}\backslash\{0\}$, we have that
\begin{eqnarray}\nonumber
\lefteqn{\|\sum_{i=1}^{N}\epsilon_{i}\lambda_{i}((c_{0}-\underline{\Delta}_{s})^{2}+\lambda_{i})^{-1}x_{i}\|_{L^{q}(0,1;X_{0}^{s})}}\\\nonumber
&=&\|\sum_{i=1}^{N}\epsilon_{i}\frac{\lambda_{i}}{2i\sqrt{\lambda_{i}}}((c_{0}-\underline{\Delta}_{s}-i\sqrt{\lambda_{i}})^{-1}-(c_{0}-\underline{\Delta}_{s}+i\sqrt{\lambda_{i}})^{-1})x_{i}\|_{L^{q}(0,1;X_{0}^{s})}\\\label{aall}
&\leq&C_{1}\|\sum_{i=1}^{N}\epsilon_{i}x_{i}\|_{L^{q}(0,1;X_{0}^{s})},
\end{eqnarray}
for a suitable constant $C_{1}>0$. Therefore, $(c_{0}-\underline{\Delta}_{s})^{2}\in\mathcal{R}(\theta)$. Furthermore, for any $c>0$, by \cite[Lemma 2.6]{RS2} we also have that $(c_{0}-\underline{\Delta}_{s})^{2}+c\in\mathcal{R}(\theta)$ and its $R$-bound can be chosen uniformly bounded in $c$.

Next, we write
\begin{eqnarray*}
\lefteqn{\|(c_{0}^{2}-2c_{0}\underline{\Delta}_{s})((c_{0}-\underline{\Delta}_{s})^{2}+c)^{-1}\|_{\mathcal{L}(X_{0}^{s})}}\\
&\leq&\|(c_{0}^{2}-2c_{0}\underline{\Delta}_{s})(c_{0}-\underline{\Delta}_{s})^{-1}\|_{\mathcal{L}(X_{0}^{s})}\|(c_{0}-\underline{\Delta}_{s})((c_{0}-\underline{\Delta}_{s})^{2}+c)^{-1}\|_{\mathcal{L}(X_{0}^{s})},
\end{eqnarray*}
so that $\|(c_{0}^{2}-2c_{0}\underline{\Delta}_{s})((c_{0}-\underline{\Delta}_{s})^{2}+c)^{-1}\|_{\mathcal{L}(X_{0}^{s})}\rightarrow 0$ as $c\rightarrow \infty$ due to \cite[Corollary 2.4]{RS2} (with $A=(c_{0}-\underline{\Delta}_{s})^{2}$, $\phi=\frac{1}{2}$, $\eta\in(0,\frac{1}{2})$ and $z=c$) and \cite[Lemma 3.6]{RS3} (with $A=c_{0}-\underline{\Delta}_{s}$ and $z=\frac{1}{2}$). Hence the result follows from $\underline{\Delta}_{s}^{2}+c=(c_{0}-\underline{\Delta}_{s})^{2}+c-c_{0}^{2}+2c_{0}\underline{\Delta}_{s}$ by perturbation, see \cite[Theorem 1]{KL}.
\end{proof}

\section{The Swift-Hohenberg equation on manifolds with conical singularities}

In this section we consider the problem \eqref{CH1}-\eqref{CH2} on a conic manifold $\mathbb{B}$. We show existence and regularity results for solutions by applying the theory of Section 2 to the Laplacian \eqref{DD}. In order to treat the non-linearity, we put some further restriction on the Mellin-Sobolev space data. Recall that $\lambda_{1}$ denotes the greatest non-zero eigenvalue of the boundary Laplacian $\Delta_{\mathfrak{h}(0)}$, i.e. the Laplacian induced on $\partial\mathcal{B}$ by the metric $\mathfrak{h}(0)$ from \eqref{metric}. Let $p,q\in(1,\infty)$ such that 
\begin{gather}\label{pq}
\frac{2}{q}<-\frac{n-1}{2}+\sqrt{\Big(\frac{n-1}{2}\Big)^2-\lambda_{1}} \quad \mbox{and} \quad \frac{2}{q}+\frac{n+1}{p}<2,
\end{gather}
and the weight $\gamma$ be chosen as
\begin{eqnarray}\label{weight}
\frac{n-3}2+\frac{2}{q}<\gamma<\min\Big\{-1+\sqrt{\Big(\frac{n-1}{2}\Big)^2-\lambda_{1}} ,\frac{n+1}2\Big\}.
\end{eqnarray}
We start with the short time existence and space asymptotics result. According to the data chosen above, we denote $X_{2}^{s}=\mathcal{D}(\underline{\Delta}_{s}^{2})$ and $X_{\frac{1}{q},q}^{s}=(X_{2}^{s},X_{0}^{s})_{\frac{1}{q},q}$.

\begin{theorem}[Short time solution]\label{Tshort}
Let $s\geq0$ and $p$, $q$, $\gamma$ be chosen as in \eqref{pq}-\eqref{weight}. Then, for any $u_{0}\in (\mathcal{D}(\underline{\Delta}_{s}^{2}),\mathcal{H}_{p}^{s,\gamma}(\mathbb{B}))_{\frac{1}{q},q}$ there exists a $T>0$ and a unique 
\begin{gather}\label{maxlqrequ}
u\in W^{1,q}(0,T;\mathcal{H}_{p}^{s,\gamma}(\mathbb{B}))\cap L^{q}(0,T;\mathcal{D}(\underline{\Delta}_{s}^{2}))
\end{gather}
solving the problem \eqref{CH1}-\eqref{CH2} on $[0,T]\times\mathbb{B}$, where the bi-Laplacian domain is described in \eqref{DD22}. In addition $u$ satisfies
\begin{gather}\label{contregu33}
u\in C^{1}((0,T);\mathcal{H}_{p}^{s,\gamma}(\mathbb{B}))\cap C((0,T);\mathcal{D}(\underline{\Delta}_{s}^{2})).
\end{gather}
\end{theorem}

\begin{proof}
We will apply the theorem of Cl\'ement and Li to \eqref{CH1}-\eqref{CH2} with $A=(\underline{\Delta}_{s}+1)^{2}$, $G(t,u)=F(t,u)$ and the Banach couple $X_{2}^{s}$, $X_{0}^{s}$. By Theorem \ref{RsecD}, let $c_{0}>0$ such that $c_{0}-\underline{\Delta}_{s}\in\mathcal{P}(\theta)$, for some $\theta\in(\frac{\pi}{2},\pi)$. Denote $c_{1}=\max\{0,c_{0}^{2}\tan^{2}(\theta)-1\}$ and note that $c>c_{1}$ implies $-c_{0}\pm i\sqrt{1+c}\in S_{\theta}$. Therefore with $c>c_{1}$ sufficiently large, by the sectoriality of $c_{0}-\underline{\Delta}_{s}$, we can make the $\mathcal{L}(X_{0}^{s})$-norm of 
$$
2\underline{\Delta}_{s}(\underline{\Delta}_{s}^{2}+1+c)^{-1}=2\underline{\Delta}_{s}(c_{0}-\underline{\Delta}_{s}-c_{0}+i\sqrt{1+c})^{-1}(c_{0}-\underline{\Delta}_{s}-c_{0}-i\sqrt{1+c})^{-1}
$$
arbitrary small. Hence, by Lemma \ref{RD2}, \cite[Lemma 2.6]{RS2} and \cite[Theorem 1]{KL}, there exists some $c>0$ such that 
\begin{gather}\label{oppA}
A+c=\underline{\Delta}_{s}^{2}+2\underline{\Delta}_{s}+1+c\in \mathcal{R}(\theta).
\end{gather}
Thus, the condition (H1) of Theorem \ref{CL} is satisfied. 

Concerning the real interpolation space $(X_{2}^{s},X_{0}^{s})_{\frac{1}{q},q}$, by \cite[Lemma 5.2]{RS2} we estimate 
\begin{gather}\label{int}
(X_{0}^{s},X_{2}^{s})_{1-\frac{1}{q},q}\hookrightarrow (X_{0}^{s},\mathcal{D}(\underline{\Delta}_{s}))_{1-\frac{1}{q},q}\hookrightarrow \bigcap_{\varepsilon>0}\mathcal{H}_{p}^{s+2-\frac{2}{q}-\varepsilon,\gamma+2-\frac{2}{q}-\varepsilon}(\mathbb{B})\oplus\mathbb{C}_{\omega}.
\end{gather}
Let $U\subset X_{\frac{1}{q},q}^{s}$ be an open bounded set such that $u_{0}\in U$. Moreover, take $u_{1},u_{2}\in U$ and $t_{1},t_{2}\in[0,T_{0}]$, for certain $T_{0}>0$. Recall that 
\begin{gather}\label{F}
F(t,u)=\sum_{k=0}^{m}\alpha_{k}(t)u^{k},
\end{gather}
for some $m\in\mathbb{N}$ and a collection of locally Lipschitz on $\mathbb{R}$ complex valued functions $\{\alpha_{k}\}_{k\in\{0,...,m\}}$. If we denote by $L_{\alpha_{k}}$ the Lipschitz bound of $\alpha_{k}|_{[0,T_{0}]}$, by \cite[Corollary 3.3]{RS2} and \eqref{int} we estimate 

\begin{eqnarray}\nonumber
\lefteqn{\|F(t_{1},u_{1})-F(t_{2},u_{2})\|_{X_{0}^{s}}}\\\nonumber
&\leq&\|\sum_{k=0}^{m}\alpha_{k}(t_{1})(u_{1}^{k}-u_{2}^{k})\|_{X_{0}^{s}}+\|\sum_{k=0}^{m}(\alpha_{k}(t_{1})-\alpha_{k}(t_{2}))u_{2}^{k}\|_{X_{0}^{s}}\\\nonumber
&\leq&C_{1}(\max_{k\in\{1,...,m\}}\sup_{\tau\in[0,T_{0}]}|\alpha_{k}(\tau)|)(\sum_{k=1}^{m}\sum_{i=0}^{k-1}\|u_{1}\|_{X_{\frac{1}{q},q}^{s}}^{k-1-i}\|u_{2}\|_{X_{\frac{1}{q},q}^{s}}^{i})\|u_{1}-u_{2}\|_{X_{\frac{1}{q},q}^{s}}\\\label{lip}
&&+C_{1}|t_{1}-t_{2}|(\max_{k\in\{0,...,m\}}L_{\alpha_{k}})\sum_{k=0}^{m}\|u_{2}\|_{X_{\frac{1}{q},q}^{s}}^{k},
\end{eqnarray}
for some constant $C_{1}>0$ depending only on $s$, $p$, $q$ and $\gamma$. Therefore, the condition (H3) of Theorem \ref{CL} is also satisfied and thus we find a solution $u$ of \eqref{CH1}-\eqref{CH2} on $[0, T]\times\mathbb{B}$ satisfying \eqref{maxlqrequ}, for some $T\in (0, T_{0}]$.

Let us now verify \eqref{contregu33}. To this end consider the linear equation
\begin{eqnarray}\label{test11a}
v'(t)+((\underline{\Delta}_{s}+1)^{2}+c)v(t)&=&e^{-ct}F(t,u(t)), \quad t\in(0,T),\\\label{test22a}
v(0)&=&u_{0}.
\end{eqnarray}
Clearly, the above equation has a solution 
\begin{gather}\label{uvp}
e^{-ct}u \quad \mbox{in} \quad W^{1,q}(0,T;X_{0}^{s})\cap L^{q}(0,T;X_{2}^{s}). 
\end{gather}

By \eqref{embmaxreg}, \eqref{int} and \cite[Lemma 3.2]{RS2}, for each $k\in\mathbb{N}$ we have that 
$$
u^{k}\in \bigcap_{\varepsilon>0}C([0,T];\mathcal{H}_{p}^{s+2-\frac{2}{q}-\varepsilon,\gamma+2-\frac{2}{q}-\varepsilon}(\mathbb{B})\oplus\mathbb{C}_{\omega}).
$$
Thus, by \eqref{F} we obtain that
\begin{gather}\label{hope}
F(\cdot,u(\cdot))\in \bigcap_{\varepsilon>0}C([0,T];\mathcal{H}_{p}^{s+2-\frac{2}{q}-\varepsilon,\gamma+2-\frac{2}{q}-\varepsilon}(\mathbb{B})\oplus\mathbb{C}_{\omega}).
\end{gather} 
Hence, by extending $F(\cdot,u(\cdot))$ to $(T,\infty)$ by constant, from Proposition \ref{propmax} the problem \eqref{test11a}-\eqref{test22a} has a unique solution $v\in W^{1,q}(0,T;X_{0}^{s})\cap L^{q}(0,T;X_{2}^{s})$, which by \eqref{uvp} satisfies $v=e^{-ct}u$. 

Since $c_{0}-\underline{\Delta}_{s}$ is sectorial, by \cite[Lemma 5.2]{RS2}, \cite[(I.2.5.2)]{Am} and \cite[(I.2.9.6)]{Am} we have that 
\begin{eqnarray*}
\lefteqn{\mathcal{H}_{p}^{s+2-\frac{2}{q}-\varepsilon,\gamma+2-\frac{2}{q}-\varepsilon}(\mathbb{B})\oplus\mathbb{C}_{\omega}}\\
&&\hookrightarrow (X_{0}^{s},\mathcal{D}(c_{0}-\underline{\Delta}_{s}))_{1-\frac{1}{q}-\varepsilon,q}\hookrightarrow (X_{0}^{s},(X_{0}^{s},\mathcal{D}((c_{0}-\underline{\Delta}_{s})^{2}))_{\frac{1}{2}-\varepsilon,q} )_{1-\frac{1}{q}-\varepsilon,q} 
\end{eqnarray*}
for all $\varepsilon>0$ sufficiently small. Therefore, by \cite[(I.2.5.2)]{Am} and reiteration \cite[(I.2.8.4)]{Am} we obtain
\begin{gather}\label{int244}
\mathcal{H}_{p}^{s+2-\frac{2}{q}-\varepsilon,\gamma+2-\frac{2}{q}-\varepsilon}(\mathbb{B})\oplus\mathbb{C}_{\omega}\hookrightarrow (X_{0}^{s},X_{2}^{s})_{\frac{1}{2}-\frac{1}{2q}-4\varepsilon,q} 
\end{gather}
for all $\varepsilon>0$ small enough. By the above embedding and \eqref{hope} we find that $F(\cdot,u(\cdot))\in C([0,T];(X_{0}^{s},X_{2}^{s})_{\frac{1}{2}-\frac{1}{2q}-4\varepsilon,q})$, for all $\varepsilon>0$ sufficiently small. Then \eqref{contregu33} follows by Proposition \ref{propmax} applied to \eqref{test11a}-\eqref{test22a}.
\end{proof}

Space asymptotics for the solution $u$ of Theorem \ref{Tshort} close to the singularity are provided by \eqref{DD22}, where the interplay with the local geometry can be seen. In addition, we have the following.

\begin{proposition}[Local asymptotics]\label{lpo}
Let $s\geq0$ and $p$, $q$, $\gamma$ be chosen as in \eqref{pq}-\eqref{weight}. Then, for the unique solution $u$ of Theorem \ref{Tshort} we have that 
\begin{gather}\label{contint}
u\in C([0,T];(X_{2}^{s},X_{0}^{s})_{\frac{1}{q},q})\hookrightarrow \bigcap_{\varepsilon>0} C([0,T];\mathcal{H}_{p}^{s+2-\frac{2}{q}-\varepsilon,\gamma+2-\frac{2}{q}-\varepsilon}(\mathbb{B})\oplus\mathbb{C}_{\omega})\hookrightarrow C([0,T];C(\mathbb{B})).
\end{gather}
If in addition $q>2$, then 
$$
(X_{2}^{s},X_{0}^{s})_{\frac{1}{q},q}\hookrightarrow\mathcal{H}_{p}^{s+2,\gamma+2}(\mathbb{B})\oplus\mathbb{C}_{\omega}\hookrightarrow C(\mathbb{B}).
$$
In both cases, the asymptotic behavior of the solution close to the singularity can be seen by \cite[Lemma 3.2]{RS2} and \eqref{weight}.
\end{proposition}
\begin{proof}
The first embedding follows by \eqref{embmaxreg} combined with \eqref{int} and \cite[Lemma 3.2]{RS2}. For the second embedding, let $c_{0}>0$ such that by Theorem \ref{RsecD} and Lemma \ref{RD2} both $c_{0}-\underline{\Delta}_{s}$ and $\underline{\Delta}_{s}^{2}+c_{0}$ are sectorial. We have that $X_{2}^{s}=\mathcal{D}((c_{0}-\underline{\Delta}_{s})^{2})$ with equivalence to the respective norms. Therefore, by \cite[(I.2.5.2)]{Am} and \cite[(I.2.9.6)]{Am} we obtain 
$$
(X_{2}^{s},X_{0}^{s})_{\frac{1}{q},q}=(X_{0}^{s},X_{2}^{s})_{1-\frac{1}{q},q}\hookrightarrow (X_{0}^{s},\mathcal{D}((c_{0}-\underline{\Delta}_{s})^{2}))_{1-\frac{1}{q},q}\hookrightarrow \mathcal{D}(c_{0}-\underline{\Delta}_{s}).
$$
\end{proof}

\begin{remark}
If the metric $\mathfrak{h}(x)$ is independent of $x$ when $x$ is close to zero, then by \cite[Theorem 3.3]{Ro2} we have more information concerning the contribution of the asymptotics space $\mathcal{F}_{\Delta,\gamma}$ defined in \eqref{DD22} to the interpolation space $(X_{2}^{s},X_{0}^{s})_{\frac{1}{q},q}$. In this case \eqref{contint} provides sharper asymptotics.
\end{remark}

Next we proceed to the main result of this section by showing that under a suitable condition, the solution given by Theorem \ref{Tshort} exists for all times. We follow a direct proof by estimating all the parameters determining the Banach fixed point step in Theorem \ref{CL}. For an alternative approach by using the maximal $L^q$-regularity property we refer to \cite[Corollary 5.1.2]{PrS}. 

\begin{theorem}[Long time solution]\label{longt}
Let $s\geq0$, $p$, $q$, $\gamma$ be chosen as in \eqref{pq}-\eqref{weight}, $u_{0}\in (\mathcal{D}(\underline{\Delta}_{s}^{2}),\mathcal{H}_{p}^{s,\gamma}(\mathbb{B}))_{\frac{1}{q},q}$
with $\mathcal{D}(\underline{\Delta}_{s}^{2})$ described in \eqref{DD22}, and $u$, $T$ be given by Theorem \ref{Tshort}. Consider the following condition:
\begin{gather}\label{kappa}
\bigg\{\begin{array}{l}
\text{There exists a positive function $K(\cdot)\in C(\mathbb{R})$ such that} \\
 \text{$\|F(\cdot,u(\cdot))\|_{L^{q}(0,T;\mathcal{H}_{p}^{s,\gamma}(\mathbb{B}))}\leq K(T)$ for all $T$ as above}.
 \end{array}
\end{gather}
Then, $T$ can be taken arbitrary large if and only if \eqref{kappa} holds. 
\end{theorem}
\begin{proof}
From the proof of Theorem \ref{Tshort} we recall \eqref{oppA} and the fact that \eqref{test11a}-\eqref{test22a} has a unique solution $v=e^{-ct}u$ in $W^{1,q}(0,T;X_{0}^{s})\cap L^{q}(0,T;X_{2}^{s})$. 

Assume that \eqref{kappa} holds. By extending $F(\cdot,u(\cdot))$ to $(T,\infty)$ by constant, from \eqref{maxreg} and \eqref{hope} applied to \eqref{test11a}-\eqref{test22a} we estimate 
\begin{eqnarray}\nonumber
\lefteqn{\|u\|_{W^{1,q}(0,T;X_{0}^{s})}+\|u\|_{L^{q}(0,T;X_{2}^{s})}}\\\nonumber
&\leq&e^{cT}(\|v\|_{W^{1,q}(0,T;X_{0}^{s})}+\|v\|_{L^{q}(0,T;X_{2}^{s})})\\\nonumber
&\leq& C_{1}e^{cT}(\|e^{-ct}F(t,u(t))\|_{L^{q}(0,T;X_{0}^{s})}+\|u_{0}\|_{(X_{2}^{s},X_{0}^{s})_{\frac{1}{q},q}})\\\label{key}
&\leq& C_{1}e^{cT}(K(T)+\|u_{0}\|_{(X_{2}^{s},X_{0}^{s})_{\frac{1}{q},q}}),
\end{eqnarray}
for some positive constant $C_{1}$ independent of $T$. Thus, due to the embedding \eqref{embmaxreg}, if $T>\delta$, for some fixed $\delta>0$, then there exists a positive function $K_{1}(\cdot)\in C(\mathbb{R})$ such that 
\begin{gather}\label{intbound}
\|u\|_{C([0,T];(X_{2}^{s},X_{0}^{s})_{\frac{1}{q},q})}\leq K_{1}(T).
\end{gather}

By \eqref{int}, \eqref{F}, \eqref{int244}, \eqref{contint} and \eqref{intbound}, together with the Banach algebra property of the space $\mathcal{H}_{p}^{s+2-\frac{2}{q}-\varepsilon,\gamma+2-\frac{2}{q}-\varepsilon}(\mathbb{B})\oplus\mathbb{C}_{\omega}$ for $\varepsilon>0$ small enough provided by \cite[Lemma 3.2]{RS2}, we obtain for such $\varepsilon>0$ that $F(\cdot,u(\cdot))\in L^{\infty}(0,T;(X_{0}^{s},X_{2}^{s})_{\frac{1}{2}-\frac{1}{2q}-\varepsilon,q})$ and furthermore
\begin{gather}\label{ppos}
\|F(\cdot,u(\cdot))\|_{L^{\infty}(0,T;(X_{0}^{s},X_{2}^{s})_{\frac{1}{2}-\frac{1}{2q}-\varepsilon,q})}\leq K_{2}(T),
\end{gather}
for some positive function $K_{2}(\cdot)\in C(\mathbb{R})$. Therefore, Proposition \ref{propmax} applied to \eqref{test11a}-\eqref{test22a} implies 
\begin{gather}\label{nn}
u(t)\in X_{2}^{s} \quad \text{for each}\quad t\in[\delta,T] \quad \text{and} \quad \|u(t)\|_{X_{2}^{s}}\leq K_{3}(T), \quad t\in[\delta,T],
\end{gather} 
for some positive function $K_{3}(\cdot)\in C(\mathbb{R})$ independent of $t$.

Let $\tau\in[\delta,T)$ and consider the following linear equation,
\begin{eqnarray}\label{test15}
h'(t)+(\underline{\Delta}_{s}+1)^{2}h(t)&=&F(t,u(\tau)),\quad t>0,\\\label{test65}
h(0)&=&u(\tau).
\end{eqnarray}
By letting $h(t)=e^{ct}g(t)$, we obtain 
\begin{eqnarray}\label{test1}
g'(t)+((\underline{\Delta}_{s}+1)^{2}+c)g(t)&=&e^{-ct}F(t,u(\tau)),\quad t>0,\\\label{test7}
g(0)&=&u(\tau).
\end{eqnarray} 
By \cite[Lemma 3.2]{RS2} and \eqref{int}, for each $T_{1}>0$ we have that 
$$
F(\cdot,u(\tau))\in \bigcap_{\varepsilon>0}C([0,T_{1}];\mathcal{H}_{p}^{s+2-\frac{2}{q}-\varepsilon,\gamma+2-\frac{2}{q}-\varepsilon}(\mathbb{B})\oplus\mathbb{C}_{\omega})\hookrightarrow L^{q}(0,T_{1};X_{0}^{s}).
$$
Hence, by Proposition \ref{propmax}, there exists a unique $g\in W^{1,q}(0,T_{1};X_{0}^{s})\cap L^{q}(0,T_{1};X_{2}^{s})$ solving \eqref{test1}-\eqref{test7}. Moreover, the function $w(t)=g(t)-e^{-tA}u(\tau)$ satisfies 
\begin{eqnarray*}
w'(t)+((\underline{\Delta}_{s}+1)^{2}+c)w(t)&=&e^{-ct}F(t,u(\tau)),\quad t>0,\\
w(0)&=&0,
\end{eqnarray*}
where $A$ is defined in \eqref{oppA}.

Fix $T_{2}>0$ and assume that $T_{1}\in(0,T_{2}]$. By \eqref{embmaxreg}, \eqref{maxreg}, \eqref{int} and \eqref{F} for any $\varepsilon>0$ small enough we have that 
\begin{eqnarray}\nonumber
\lefteqn{\|w\|_{C([0,T_{1}];(X_{2}^{s},X_{0}^{s})_{\frac{1}{q},q})}}\\\nonumber
&\leq &C_{2}\|w\|_{W^{1,q}(0,T_{1};X_{0}^{s})\cap L^{q}(0,T_{1};X_{2}^{s})}\leq C_{3}\|e^{-c(\cdot)}F(\cdot,u(\tau))\|_{L^{q}(0,T_{1};X_{0}^{s})}\\\label{jjaa}
&\leq &C_{4}T_{1}^{\frac{1}{q}}(\max_{k\in\{0,...,m\}}\sup_{t\in[0,T_{1}]}|\alpha_{k}(t)|)\sum_{k=0}^{m}\|u(\tau)\|_{(X_{2}^{s},X_{0}^{s})_{\frac{1}{q},q}}^{k},
\end{eqnarray}
for some positive constants $C_{2}$, $C_{3}$ and $C_{4}$ independent of $\tau$, $T$ and $T_{1}$, where we have used the Banach algebra property of $\mathcal{H}_{p}^{s+2-\frac{2}{q}-\varepsilon,\gamma+2-\frac{2}{q}-\varepsilon}(\mathbb{B})\oplus\mathbb{C}_{\omega}$ due to \cite[Lemma 3.2]{RS2}. 

By \eqref{nn}, the function $t\mapsto f(t)=u(\tau)-e^{-tA}u(\tau)$ satisfies $f'(t)+((\underline{\Delta}_{s}+1)^{2}+c)f(t)=((\underline{\Delta}_{s}+1)^{2}+c)u(\tau)$, $t>0$, and $f(0)=0$. Hence, by \eqref{embmaxreg}, \eqref{maxreg} and \eqref{nn}, we obtain that 
\begin{gather}\label{pppl}
\|(e^{-(\cdot)A}-I)u(\tau)\|_{C([0,T_{1}];(X_{2}^{s},X_{0}^{s})_{\frac{1}{q},q})}\leq K_{4}(T)T_{1}^{\frac{1}{q}},
\end{gather}
for some positive function $K_{4}(\cdot)\in C(\mathbb{R})$ independent of $\tau$ and $T_{1}$. 

Returning to the solution of \eqref{test15}-\eqref{test65} we estimate
\begin{eqnarray}\nonumber
\lefteqn{\phi(T_{1})=\|h(\cdot)-u(\tau)\|_{C([0,T_{1}];(X_{2}^{s},X_{0}^{s})_{\frac{1}{q},q})}}\\\nonumber
&\leq&e^{cT_{1}}\|g(\cdot)-e^{-c(\cdot)}u(\tau)\|_{C([0,T_{1}];(X_{2}^{s},X_{0}^{s})_{\frac{1}{q},q})}\\\nonumber
&\leq&e^{cT_{1}}\|w\|_{C([0,T_{1}];(X_{2}^{s},X_{0}^{s})_{\frac{1}{q},q})}+e^{cT_{1}}\|(e^{-(\cdot)A}-I)u(\tau)\|_{C([0,T_{1}];(X_{2}^{s},X_{0}^{s})_{\frac{1}{q},q})}\\\nonumber
&&+(e^{cT_{1}}-1)\|u(\tau)\|_{(X_{2}^{s},X_{0}^{s})_{\frac{1}{q},q}}.
\end{eqnarray}
Therefore, by \eqref{intbound}, \eqref{jjaa} and \eqref{pppl} we find that $\phi(T_{1})\rightarrow0$ as $T_{1}\rightarrow0$ uniformly in $\tau\in[\delta,T)$ and $T$ lying in a bounded set. 

We further estimate
\begin{eqnarray}\nonumber
\lefteqn{\psi(T_{1})=\|h\|_{W^{1,q}(0,T_{1};X_{0}^{s})\cap L^{q}(0,T_{1};X_{2}^{s})}}\\\nonumber
&\leq&C_{5}e^{cT_{1}}(\|w\|_{W^{1,q}(0,T_{1};X_{0}^{s})\cap L^{q}(0,T_{1};X_{2}^{s})}\\\nonumber
&&+\|e^{-(\cdot)A}u(\tau)\|_{L^{q}(0,T_{1};X_{0}^{s})}+2\|e^{-(\cdot)A}Au(\tau)\|_{L^{q}(0,T_{1};X_{0}^{s})})\\\nonumber
&\leq&C_{5}e^{cT_{1}}(\|w\|_{W^{1,q}(0,T_{1};X_{0}^{s})\cap L^{q}(0,T_{1};X_{2}^{s})}\\\label{ooppl}
&&+\|e^{-(\cdot)A}\|_{L^{q}(0,T_{1};\mathcal{L}(X_{0}^{s}))}(\|u(\tau)\|_{X_{0}^{s}}+2\|Au(\tau)\|_{X_{0}^{s}})),
\end{eqnarray}
for some constant $C_{5}>0$ independent of $\tau$, $T$ and $T_{1}$. Therefore, by \eqref{intbound}, \eqref{nn}, \eqref{jjaa} and Lemma \ref{estimate1} we also find that $\psi(T_{1})\rightarrow0$ as $T_{1}\rightarrow0$ uniformly in $\tau\in[\delta,T)$ and $T$ lying in a bounded set.

Fix $r_{0}>0$ and assume that $r\in(0,r_{0}]$. Denote by $\Sigma_{r,T_{1}}$ the set of functions $a\in W^{1,q}(0,T_{1};X_{0}^{s})\cap L^{q}(0,T_{1};X_{2}^{s})$ satisfying $a(0)=u(\tau)$ and $\|a-h\|_{W^{1,q}(0,T_{1};X_{0}^{s})\cap L^{q}(0,T_{1};X_{2}^{s})}\leq r$. Furthermore, let the map $\gamma_{r,T_{1}}:\Sigma_{r,T_{1}}\rightarrow \Sigma_{r,T_{1}}$ defined by $\gamma_{r,T_{1}}(a)=b$, where $b$ is the unique in $W^{1,q}(0,T_{1};X_{0}^{s})\cap L^{q}(0,T_{1};X_{2}^{s})$ solution of 
\begin{eqnarray*}
b'(t)+(\underline{\Delta}_{s}+1)^{2}b(t)&=&F(t,a(t)), \quad t\in(0,T_{1}),\\
b(0)&=&u(\tau).
\end{eqnarray*}

Note that the existence of $b$ follows since $F(\cdot,a(\cdot))\in L^{q}(0,T_{1};X_{0}^{s})$ due to \eqref{embmaxreg}, \cite[Lemma 3.2]{RS2} and \eqref{contint}. The proof of Theorem \ref{CL} is based on the Banach fixed point theorem applied to $\gamma_{r,T_{1}}$. Namely, it is shown that for sufficiently small $r$ and $T_{1}$, $\gamma_{r,T_{1}}$ becomes a well defined map which is also a contraction. The constant bounds that determine the above two properties of $\gamma_{r,T_{1}}$ are estimated in \cite[(2.11)]{CL} and \cite[(2.14)]{CL} and are given by \cite[(2.13)]{CL} and \cite[(2.17)]{CL} respectively; note that in our case \cite[(2.13)]{CL} and \cite[(2.17)]{CL} are simplified due to the semilinearity. In general, they are determined by the following five parameters, namely: (i) the maximal $L^{q}$-regularity bound $M$ from \cite[Corollary 2.3]{CL}, (ii) the Lipschitz bound $L$ in \cite[(2.9)-(2.10)]{CL}, (iii) the initial data $u(0)$ in $(X_{2}^{s},X_{0}^{s})_{\frac{1}{q},q}$, (iv) the decay of $\phi(z)$ as $z\rightarrow0$ and (v) the decay of $\psi(z)$ as $z\rightarrow0$.

Take $u(\tau)$ as initial data for \eqref{CH1}-\eqref{CH2} with $\tau$ sufficiently close to $T$. Denote by $T_{\max}$ the supremum of all possible $T$ and assume that $T_{\max}<\infty$. The bound $M$ in our case is determined by \eqref{maxreg2} applied to the operator $(\underline{\Delta}_{s}+1)^{2}$, as well as by the norm of the embedding \eqref{embmaxreg}, and hence stays uniformly bounded in $\tau$, $T$ and $T_{1}$. The constant $\|u(\tau)\|_{(X_{2}^{s},X_{0}^{s})_{\frac{1}{q},q}}$ also stays uniformly bounded in $\tau$, $T$ and $T_{1}$ due to \eqref{intbound}. Furthermore, by \eqref{lip}, the Lipschitz constant $L$ is determined by the $(X_{2}^{s},X_{0}^{s})_{\frac{1}{q},q}$-norm of the solution $u$, and hence it is uniformly bounded in $\tau$, $T$ and $T_{1}$ as well. Next, the right hand side of \cite[(2.8)]{CL} becomes arbitrary small uniformly in $\tau$, $T$ and $T_{1}$, by taking $r_{0}$ and $T_{2}$ sufficiently small. Finally, both $\phi(T_{1})$ and $\psi(T_{1})$ become arbitrary small uniformly in $\tau$ and $T$, by taking the new time step $T_{1}$ sufficiently small. 

Therefore, by taking $T$ sufficiently close to $T_{\max}$, we can repeat the Cl\'ement-Li Banach fixed point step of Theorem \ref{Tshort} in order to obtain a $W^{1,q}(T_{\max}-\rho,T_{\max}+\rho;X_{0}^{s})\cap L^{q}(T_{\max}-\rho,T_{\max}+\rho;X_{2}^{s})$ solution of \eqref{CH1}-\eqref{CH2}, for some $\rho\in(0,T_{\max})$. The restriction of such a solution is also unique in $W^{1,q}(T_{\max}-\rho,T_{\max}-\eta;X_{0}^{s})\cap L^{q}(T_{\max}-\rho,T_{\max}-\eta;X_{2}^{s})$ for each $\eta\in(0,\rho)$, and therefore it coincides with $u$ in $(T_{\max}-\rho,T_{\max})$, which contradicts the maximality of $T_{\max}$.

Now if $T$ in Theorem \ref{Tshort} can be taken arbitrary large, then \eqref{kappa} holds due to \eqref{hope}.
\end{proof}

\section{The Swift-Hohenberg equation on closed manifolds} 

In this section we study the Swift-Hohenberg equation on closed manifolds by employing the singular analysis results described in the previous sections. The main idea of this approach is based on geodesic polar coordinates. Let $\mathcal{M}$ be a smooth closed connected $(n+1)$-dimensional manifold, endowed with a Riemannian metric $\mathfrak{f}$. Denote $\mathbb{M}=(\mathcal{M},\mathfrak{f})$, $\mathbb{S}^{n}=\{z\in\mathbb{R}^{n+1}\, |\, |z|=1\}$ the unit sphere, and let $o\in \mathcal{M}$ be an arbitrary point. If $z\in \mathcal{M}\backslash\{o\}$, let $x=d(o,z)$ be the geodesic distance between $o$ and $z$, where $d$ is the metric distance induced by $\mathfrak{f}$. Then, there exists an $r>0$ such that $(x,y)\in (0,r)\times \mathbb{S}^{n}$ are local coordinates around $o$ and the metric in these coordinates admits the structure $\mathfrak{f}=dx^{2}+x^{2}\mathfrak{f}_{\mathbb{S}^{n}}(x)$, where $x\mapsto \mathfrak{f}_{\mathbb{S}^{n}}(x)$ is a smooth family of Riemannian metrics on $\mathbb{S}^{n}$, see, e.g., \cite[Lemma 5.5.7]{Pa}. Assume that $x\mapsto \mathfrak{f}_{\mathbb{S}^{n}}(x)$ is smooth up to $x=0$.

According to the above setting we consider the problem \eqref{CH1}-\eqref{CH2} on the conic manifold $\mathbb{B}=((\mathcal{M}\backslash\{o\})\cup(\{0\}\times \mathbb{S}^{n}),\mathfrak{f})$. Denote by $u$ the unique short (respectively long) time solution on $[0,T]\times\mathbb{B}$, $T>0$, subject to appropriate initial data $u_{0}$, given by Theorem \ref{Tshort} (respectively Theorem \ref{longt}). By the embedding \eqref{embmaxreg}, \cite[Lemma 3.2]{RS2} and Proposition \ref{lpo} we obtain the following.

\begin{corollary}\label{corpatr}
The above solution satisfies $u\in C([0,T];C(\mathbb{M}))$ and moreover there exist some $\alpha\in C([0,T];\mathbb{C})$ and $C>0$ that only depend on $o$, $u_{0}$ and $o$, $u_{0}$, $n$, $T$, $r$ respectively, such that
\begin{gather}\label{betadec}
|u(x,y,t)-\alpha(t)|\leq C x^{\beta}, \quad x\in[0,r), \, y\in \mathbb{S}^{n}, \, t\in[0,T],
\end{gather}
for some $\beta>0$ that only depends on $u_{0}$, $n$, and $\lambda_{1}$, where $\lambda_{1}$ is the greatest non-zero eigenvalue of the Laplacian on $\mathbb{S}^{n}$ induced by the metric $\mathfrak{f}_{\mathbb{S}^{n}}(0)$.
\end{corollary}

For the dependence of $\lambda_{1}$ on the local geometry we refer to \cite{LY}, \cite{LY2} and \cite{Li}. Moreover, $\beta$ can be chosen arbitrary close to $2$ provided that $-\lambda_{1}$ is large enough, as it can be seen by the following. 

\begin{example}\label{ex123}
Assume that in Corollary \ref{corpatr} the initial data $u_{0}$ and the greatest non-zero eigenvalue $\lambda_{1}$ satisfy $u_{0}={\text constant}$ on $\mathcal{M}$ and $-\lambda_{1}\geq 2(n+1)$ respectively. Then, for any $\beta<2$ there exists a $C>0$ such that \eqref{betadec} holds.
\end{example}
\begin{proof}
For any $\varepsilon\in(0,1)$ in Theorem \ref{Tshort} (respectively Theorem \ref{longt}) we can choose $s\geq0$ arbitrary, $\gamma=\frac{n+1}{2}-\varepsilon$, $p=\frac{n+1}{2-\varepsilon}$ and $q>\frac{2}{\varepsilon}$ arbitrary. 
\end{proof}

\section{Appendix}

We include here the proof of Proposition \ref{propmax}. We start with the following elementary boundedness property of holomorphic semigroups, that can be found in \cite[Section IV.2.1]{Am}.

\begin{lemma}\label{estimate1}
Let $A\in \mathcal{P}(K,\theta)$ in $X_{0}$ for certain $K\geq1$ and $\theta>\frac{\pi}{2}$. Then, $A$ generates a bounded holomorphic semigroup $\{e^{-zA}\}_{z\in S_{\phi}}$ on $X_{0}$, where $\phi\in[0,\theta-\frac{\pi}{2})$. Furthermore, for any $a\in[0,1)$ we have that $t\mapsto e^{-tA} \in C((0,\infty);\mathcal{L}(X_{0},\mathcal{D}(A^{a}))$ and 
$$
\|A^{a}e^{-tA}\|_{\mathcal{L}(X_{0})}\leq C t^{-a}, \quad t>0,
$$
for some constant $C>0$ that only depends on $K$, $\theta$, and $a$.
\end{lemma}

\subsubsection*{Proof of Proposition 2.7}
Consider the following two linear parabolic problems
\begin{eqnarray}\label{APF1}
u_{1}'(t)+Au_{1}(t)&=&g(t), \quad t\in(0,T),\\\label{abp1}
u_{1}(0)&=&0 
\end{eqnarray}
and
\begin{eqnarray}\label{APF2}
u_{2}'(t)+Au_{2}(t)&=&0, \quad t\in(0,T),\\\label{abp2}
u_{2}(0)&=&u_{0}.
\end{eqnarray}

For any $T>0$, by Theorem \ref{KaW}, the problem \eqref{APF1}-\eqref{abp1} admits a unique solution $u_{1}\in W^{1,q}(0,T;X_{0})\cap L^{q}(0,T;X_{1})$. More precisely, let us denote again by $A$ the natural extension of $A$ from $X_{0}$ to $L^{q}(0,T;X_{0})$, i.e. $(Af)(t)=Af(t)$ for almost all $t$. By Kahane's inequality, $A$ with domain $L^{q}(0,T;X_{1})$ in $L^{q}(0,T;X_{0})$ is also $R$-sectorial of angle $\theta$. 

Let $B:u\mapsto \partial_{t}u$ with domain $\mathcal{D}(B)=\{u\in W^{1,q}(0,T;X_{0})\, |\, u(0)=0\}$ in $L^{q}(0,T;X_{0})$. Recall that $\sigma(B)=\emptyset$ and for each $\lambda\in\mathbb{C}$
\begin{gather}\label{resB}
(B+\lambda)^{-1}g(t)=\int_{0}^{t}e^{(s-t)\lambda} g(s)ds, \quad g\in L^{q}(0,T;X_{0}).
\end{gather}
Since our underline Banach space fulfills the UMD property, by \cite[Theorem 8.5.8]{Ha}, for any $\phi\in(0,\frac{\pi}{2})$ the operator $B$ has bounded $H^\infty$-calculus of angle $\phi$ (see, e.g., \cite[Definition 3.3.12]{PrS} for the definition of this property). Furthermore, if in the definition formula for the $H^\infty$-calculus for $B$ (see, e.g., \cite[(3.29)]{PrS}) we replace $(B+\lambda)^{-1}g(t)$ by $\int_{0}^{t}e^{(s-t)\lambda} g(s)\psi_{[0,T]}(s)ds$, where $\psi_{[0,T]}$ denotes the characteristic function on $[0,T]$, then by \cite[Corollary 8.5.3 (a)]{Ha} we find that the $H^\infty$-bound of $B$ (i.e. the constant bound in \cite[(3.60)]{PrS}) for arbitrary $T$, is smaller than the $H^\infty$-bound of $B$ for the case of $T=\infty$. Therefore, we conclude that the $H^\infty$-bound of $B$ is uniformly bounded in $T\in(0,\infty)$.

The operators $A$ and $B$ are resolvent commuting in the sense of \cite[(III.4.9.1)]{Am}. Hence, by \cite[Theorem 6.3]{KW1} the sum $A+B$ with domain $W^{1,q}(0,T;X_{0})\cap L^{q}(0,T;X_{1})$ in $L^{q}(0,T;X_{0})$ is closed and invertible. By \cite[Theorem 3.7]{DG} (or alternatively by \cite[Theorem 2.1]{Ro}) its inverse is given by 
\begin{gather}\label{a+b}
(A+B)^{-1}=\frac{1}{2\pi i}\int_{\Gamma_{\theta}}(A+\lambda)^{-1}(B-\lambda)^{-1}d\lambda \in \mathcal{L}(X_{0}).
\end{gather}
Furthermore, by the estimates in the proof of \cite[Theorem 4.4]{KW1} (alternatively see \cite[Theorem 2.1]{HH} or the proof of \cite[Theorem 3.1]{PS} or \cite[Lemma 3.4]{Ro1}), the $\mathcal{L}(L^{q}(0,T;X_{0}))$-norm of $B(A+B)^{-1}$ can be estimated by the $H^\infty$-bound of $B$ and the $R$-bound of $A$, and hence it is uniformly bounded in $T$. Therefore we have
\begin{eqnarray}\label{est1}
\|u_{1}\|_{W^{1,q}(0,T;X_{0})}+\|u_{1}\|_{ L^{q}(0,T;X_{1})}\leq C_{1}\|g\|_{L^{q}(0,T;X_{0})},
\end{eqnarray}
for some constant $C_{1}>0$ that only depends on $A$, $R$, $\theta$, and $q$.

Concerning the problem \eqref{APF2}, by characterization \cite[Theorem 3.7.11]{ABHN}, $-A$ generates a bounded holomorphic semigroup on $X_{0}$ which on the positive real semiaxis is given by the formula \cite[(3.46)]{ABHN}. Since the integral $\int_{\Gamma_{\theta}}e^{\lambda t}A(A+\lambda)^{-1}d\lambda$, $t>0$, converges absolutely, by \cite[(3.46)]{ABHN} we have that $e^{-tA}$, $t>0$, maps $X_{0}$ to $\mathcal{D}(A)$ and moreover by the dominated convergence theorem
\begin{gather}\label{rgt}
Ae^{-tA}=\frac{1}{2\pi i }\int_{\Gamma_{\theta}}e^{\lambda t}A(A+\lambda)^{-1}d\lambda \in C((0,\infty);\mathcal{L}(X_{0})).
\end{gather}

Therefore, if $w\in (X_{1},X_{0})_{\frac{1}{q},q}$ and $\sigma=\frac{q-1}{2q}$, by \cite[(I.2.5.2)]{Am}, \cite[(I.2.9.6)]{Am} and \cite[Theorem III.4.6.5]{Am} we have that
$Ae^{-tA}w=A^{1-\sigma}e^{-tA}A^{\sigma}w$. Thus, for any $\tau>1$ by Lemma \ref{estimate1} we estimate
\begin{eqnarray}\nonumber
\lefteqn{\|Ae^{-(\cdot)A}w\|_{L^{q}(1,\tau;X_{0})}}\\\label{ool}
&&\leq\|A^{1-\sigma}e^{-(\cdot)A}\|_{L^{q}(1,\tau;\mathcal{L}(X_{0}))}\|A^{\sigma}w\|_{X_{0}}\leq C_{2}(1-\tau^{\frac{1-q}{2}})^{\frac{1}{q}}\|w\|_{(X_{1},X_{0})_{\frac{1}{q},q}},
\end{eqnarray}
with some constant $C_{2}>0$ that only depends on $A$, $K$, $\theta$, and $q$.

Let the space $D_{A}(1-\frac{1}{q},q)=\{w\in X_{0}\, |\, t\mapsto Ae^{-tA}w\in L^{q}(0,1;X_{0})\}$
endowed with the norm $\|w\|_{D_{A}(1-\frac{1}{q},q)}=\|w\|_{X_{0}}+\|Ae^{-(\cdot)A}w\|_{ L^{q}(0,1;X_{0})}$. By the characterization result \cite[Proposition 2.2.2]{Lu}, we have that $D_{A}(1-\frac{1}{q},q)=(X_{1},X_{0})_{\frac{1}{q},q}$ up to norm equivalence. Hence, there exists some constant $C_{3}>0$ depending only on $A$ and $q$ such that
\begin{gather}\label{estint}
\|Ae^{-(\cdot)A}w\|_{ L^{q}(0,1;X_{0})}\leq C_{3}\|w\|_{(X_{1},X_{0})_{\frac{1}{q},q}}
\end{gather}
for any $w\in (X_{1},X_{0})_{\frac{1}{q},q}$. 

By \eqref{ool}-\eqref{estint} we conclude that $t\mapsto u_{2}(t)=e^{-tA}u_{0}$ is a $W^{1,q}(0,T;X_{0})\cap L^{q}(0,T;X_{1})$-solution of \eqref{APF2}-\eqref{abp2} and moreover
\begin{gather}\label{est2}
\|u_{2}\|_{W^{1,q}(0,T;X_{0})}+\|u_{2}\|_{ L^{q}(0,T;X_{1})}\leq C_{4}\|u_{0}\|_{(X_{1},X_{0})_{\frac{1}{q},q}},
\end{gather}
for some constant $C_{4}>0$ that only depends on $A$, $K$, $\theta$, and $q$.

Returning to the original problem, note that $u=u_{1}+u_{2}$ is a $W^{1,q}(0,T;X_{0})\cap L^{q}(0,T;X_{1})$-solution of \eqref{APF}-\eqref{thr} and the required estimate follows by \eqref{est1} and \eqref{est2} together with the fact that $2\max\{R\, , \, \sup_{\lambda\in S_{\theta}, |\lambda|\leq1}\|(A+\lambda)^{-1}\|_{\mathcal{L}(X_{0})}\}$ is a sectorial bound of $A$. The uniqueness of $u$ follows by the injectivity of $A+B$.

Concerning the estimate \eqref{x1bound}, if $\alpha\in(0,1)$, by \eqref{a+b}, \cite[(III.4.7.13)]{Am}, Cauchy's theorem and Fubini's theorem we find that
\begin{eqnarray}\label{oopp}
(A+B)^{-1}A^{-\alpha}=\frac{1}{2\pi i}\int_{\Gamma_{\theta}}(-\lambda)^{-\alpha}(A+\lambda)^{-1}(B-\lambda)^{-1}d\lambda,
\end{eqnarray}
see also \cite[(2.3)]{Ro}. Moreover, if $g\in \cap_{\tau>0}L^{\infty}(0,\tau;(X_{0},X_{1})_{\phi,\eta})$ for some $\phi\in(0,1)$ and $\eta\in(1,\infty)$, then by restricting $a\in(0,\phi)$ and taking into account \cite[(I.2.5.2)]{Am} and \cite[(I.2.9.6)]{Am} we have that $g\in \cap_{\tau>0}L^{\infty}(0,\tau;\mathcal{D}(A^{\alpha}))$. Therefore, by writing $u_{1}=(A+B)^{-1}A^{-\alpha}A^{\alpha}g$, from \eqref{oopp}, \eqref{resB} and Fubini's theorem, after changing variables we obtain
\begin{eqnarray}\label{mm}
u_{1}(t)=\frac{1}{2\pi i}\int_{0}^{t}(\int_{\Gamma_{\theta}}(-z)^{-\alpha}e^{z}(A+\frac{z}{s})^{-1}dz)s^{\alpha-1}A^{\alpha}g(t-s)ds, \quad t>0.
\end{eqnarray}

Since the integral $\int_{\Gamma_{\theta}}(-z)^{-\alpha}e^{z}A(A+\frac{z}{\rho})^{-1}dz$ converges absolutely uniformly in $\rho>0$, by the resolvent formula and the dominated convergence theorem we deduce that the map $\rho\mapsto \int_{\Gamma_{\theta}}(-z)^{-\alpha}e^{z}A(A+\frac{z}{\rho})^{-1}dz$ belongs to $C((0,\infty);\mathcal{L}(X_{0}))\cap L^{\infty}(0,\infty;\mathcal{L}(X_{0}))$. Therefore, \eqref{mm} implies 
$$
u_{1}(t)=\frac{1}{2\pi i}\int_{0}^{t}A^{-1}(\int_{\Gamma_{\theta}}(-z)^{-\alpha}e^{z}A(A+\frac{z}{s})^{-1}dz)s^{\alpha-1}A^{\alpha}g(t-s)ds, \quad t>0. 
$$ 
We conclude that $u_{1}(t)\in X_{1}$ for each $t>0$ and moreover
\begin{gather}\label{aindom}
Au_{1}(t)=\frac{1}{2\pi i}\int_{0}^{t}(\int_{\Gamma_{\theta}}(-z)^{-\alpha}e^{z}A(A+\frac{z}{s})^{-1}dz)s^{\alpha-1}A^{\alpha}g(t-s)ds, \quad t>0.
\end{gather}
Furthermore, for any $t>0$ we estimate
\begin{eqnarray}\nonumber
\lefteqn{\|Au_{1}(t)\|_{X_{0}}}\\\nonumber
&\leq&\frac{1}{2\pi }(\int_{0}^{t}\|\int_{\Gamma_{\theta}}(-z)^{-\alpha}e^{z}A(A+\frac{z}{s})^{-1}dz\|_{\mathcal{L}(X_{0})}s^{\alpha-1}ds)\|A^{\alpha}g\|_{L^{\infty}(0,t;X_{0})}\\\label{mm32}
&\leq&C_{5}t^{\alpha}\|g\|_{L^{\infty}(0,t;(X_{0},X_{1})_{\phi,\eta})},
\end{eqnarray}
for some constant $C_{5}>0$ depending only on $A$, $K$, $\theta$, $\phi$, $\eta$ and $\alpha$.

Concerning the part $u_2$, by \eqref{rgt} we have $u_{2}\in C((0,\infty);X_{1})$. By taking $\gamma\in(0,1-\frac{1}{q})$, due to \cite[(I.2.5.2)]{Am}, \cite[(I.2.9.6)]{Am} and \cite[Theorem III.4.6.5]{Am}, for any $t>0$ we can write $Au_{2}(t)=A^{1-\gamma}e^{-tA}A^{\gamma}u_{0}$. Therefore, by Lemma \ref{estimate1} we estimate $\|Au_{2}(t)\|_{X_{0}}\leq C_{6}t^{\gamma-1}\|u_{0}\|_{(X_{1},X_{0})_{\frac{1}{q},q}}$, $t>0$, for some constant $C_{6}>0$ depending only on $A$, $K$, $\theta$, $q$, and $\gamma$, so that \eqref{x1bound} follows by \eqref{mm32}.

If $g\in C([0,\infty);(X_{0},X_{1})_{\phi,\eta})$, then \cite[(I.2.5.2)]{Am} and \cite[(I.2.9.6)]{Am} imply $g\in C([0,\infty);\mathcal{D}(A^{\alpha}))$. Hence, by \eqref{aindom} we deduce $u_{1}\in C((0,\infty);X_{1})$ and \eqref{contregabs} follows by \eqref{AP}. \mbox{\ } \hfill $\square$


\begin{thebibliography}{99}

\bibitem{Am} H. Amann. {\em Linear and quasilinear parabolic problems. Vol. I. Abstract linear theory}. Monographs in Mathematics {\bf 89}, Birkh\"auser Verlag (1995).

\bibitem{ABHN} W. Arendt, C. J. K. Batty, M. Hieber, F. Neubrander. {\em Vector-valued Laplace transforms and Cauchy problems}. Monographs in Mathematics {\bf 96}, Birkh\"auser Verlag (2001).

\bibitem{CL} P. Cl\'ement, S. Li, {\em Abstract parabolic quasilinear equations and application to a groundwater flow problem}. Adv. Math. Sci. Appl. {\bf 3}, Special Issue, 17--32 (1993/94).

\bibitem{CP} P.\ Cl\'ement, J.\ Pr\"uss. {\em An operator-valued transference principle and maximal regularity on vector-valued $L_p$-spaces}. In: G. Lumer and L. Weis (eds.), Proc. of the 6th. International Conference on Evolution Equations, Marcel Dekker (2001).

\bibitem{CH} M. Cross, P. Hohenberg. {\em Pattern formation outside of equilibrium}. Rev. Mod. Phys. {\bf 65}, no. 3, 851--1123 (1993).

\bibitem{DG} G. Da Prato, P. Grisvard. {\em Sommes d' op\'erateurs lin\'eaires et \'equations diff\'erentielles op\'erationnelles}. J. Math. Pures Appl. (9) {\bf 54}, no. 3, 305--387 (1975).

\bibitem{DV} G. Dore, A. Venni. {\em On the closedness of the sum of two closed operators}. Math. Z. {\bf 196}, no. 2, 189--201 (1987). 

\bibitem{GM} J. Gil, G. Mendoza. {\em Adjoints of elliptic cone operators}. Amer. J. Math. {\bf 125}, no. 2, 357--408 (2003).

\bibitem{GKM} J. Gil, T. Krainer, G. Mendoza. {\em Resolvents of elliptic cone operators}. J. Funct. Anal. {\bf241}, no. 1, 1--55 (2006).

\bibitem{Ha} M. Haase. {\em The functional calculus for sectorial operators}. Operator Theory: Advances and Applications {\bf169}, Birkh\"auser Verlag (2006).

\bibitem{HH} R. Haller-Dintelmann, M. Hieber. {\em $H^{\infty}$-calculus for products of non-commuting operators}. Math. Z. {\bf 251}, no. 1, 85--100 (2005).

\bibitem{KS} M. Kaip, J. Saal. {\em The permanence of R-boundedness and property $(\alpha)$ under interpolation and applications to parabolic systems}. J. Math. Sci. Univ. Tokyo {\bf 19}, no. 3, 359--407 (2012).

\bibitem{KW1} N. Kalton, L. Weis. {\em The $H^{\infty}$-calculus and sums of closed operators}. Math. Ann. {\bf 321}, no. 2, 319--345 (2001).

\bibitem{KL} P. C. Kunstmann, L. Weis. {\em Perturbation theorems for maximal $L_p$-regularity}. Ann. Scuola Norm. Sup. Pisa Cl. Sci. (4) {\bf 30}, no. 2, 415--435 (2001).

\bibitem{LY} P. Li, S. T. Yau. {\em Estimates of eigenvalues of a compact Riemannian manifold}. AMS Proc. Symp. Pure Math. {\bf 36}, 205--239 (1980).

\bibitem{LY2} P. Li, S. T. Yau. {\em On the parabolic kernel of the Schr\"odinger operator}. Acta Math. {\bf 156}, no. 3-4, 153--201 (1986).

\bibitem{Li} A. Lichnerowicz. {\em G\'eometrie des groupes de transformations}. Travaux et Recherches Math\'ematiques III, Dunod (1958).

\bibitem{Le} M. Lesch. {\em Operators of Fuchs type, conical singularities, and asymptotic methods}. Teubner-Texte zur Mathematik {\bf136}, Teubner Verlag (1997).

\bibitem{Lu} A. Lunardi. {\em Analytic semigroups and optimal regularity in parabolic problems}. Modern Birkh\"auser Classics, Birkh\"auser Verlag (2012).

\bibitem{Mi} A. Mielke. {\em Instability and stability of rolls in the Swift-Hohenberg equation}. Commun. Math. Phys. {\bf 189}, no. 3, 829--853 (1997).

\bibitem{PT} L. A. Peletier, W. C. Troy. {\em Spatial patterns. Higher order models in physics and mechanics}. Progress in Nonlinear Differential Equations and Their Applications {\bf 45}, Birkh\"auser Verlag (2001).

\bibitem{Pa} P. Petersen. {\em Riemannian geometry}. Graduate Texts in Mathematics {\bf 171}, Springer Verlag (2016).

\bibitem{PM} Y. Pomeau, P. Manneville. {\em Wavelength selection in cellular flows}. Phys. Lett. A {\bf 75}, no. 4, 296--298 (1980).

\bibitem{PS} J. Pr\"uss, G. Simonett. {\em $H^\infty$-calculus for the sum of non-commuting operators}. Trans. Amer. Math. Soc. {\bf 359}, no. 8, 3549--3565 (2007).

\bibitem{PrS} J. Pr{\"u}ss, G. Simonett. {\em Moving interfaces and quasilinear parabolic evolution equations}. Monographs in Mathematics {\bf 105}, Birkh\"auser Verlag (2016).

\bibitem{Ro1} N. Roidos. {\em Closedness and invertibility for the sum of two closed operators}. Adv. Oper. Theory {\bf 3}, no. 3, 582--605 (2018).

\bibitem{Ro2} N. Roidos. {\em Complex powers for cone differential operators and the heat equation on manifolds with conical singularities}. Proceedings of the Amer. Math. Soc. {\bf 146}, no. 7, 2995--3007 (2018).

\bibitem{Ro} N. Roidos. {\em On the inverse of the sum of two sectorial operators}. J. Funct. Anal. {\bf 265}, no. 2, 208--222 (2013).

\bibitem{RS0} N. Roidos, E. Schrohe. {\em Bounded imaginary powers of cone differential operators on higher order Mellin-Sobolev spaces and applications to the Cahn-Hilliard equation}. J. Differential Equations {\bf 257}, no. 3, 611--637 (2014).

\bibitem{RS2} N. Roidos, E. Schrohe. {\em Existence and maximal $L^{p}$-regularity of solutions for the porous medium equation on manifolds with conical singularities}. Comm. Partial Differential Equations {\bf 41}, no. 9, 1441--1471 (2016).

\bibitem{RS3} N. Roidos, E. Schrohe. {\em The Cahn-Hilliard equation and the Allen-Cahn equation on manifolds with conical singularities}. Comm. Partial Differential Equations {\bf 38}, no. 5, 925--943 (2013).

\bibitem{SS1} E. Schrohe, J. Seiler. {\em Bounded $H_{\infty}$-calculus for cone differential operators}. J. Evol. Equ. {\bf 18}, no. 3, 1395--1425 (2018).

\bibitem{SS2} E. Schrohe, J. Seiler. {\em Ellipticity and invertibility in the cone algebra on $L_{p}$-Sobolev spaces}. Integr. Equ. Oper. Theory {\bf 41}, no. 1, 93--114 (2001).

\bibitem{Sh} E. Schrohe, J. Seiler. {\em The resolvent of closed extensions of cone differential operators}. Can. J. Math. {\bf 57}, no. 4, 771--811 (2005).

\bibitem{Schu} B. Schulze. {\em Pseudo-differential operators on manifolds with singularities}. Studies in Mathematics and Its Applications {\bf 24}, North Holland (1991).

\bibitem{Sei} J. Seiler. {\em The cone algebra and a kernel characterization of Green operators}. Approaches to Singular Analysis. Operator Theory: Advances and Applications {\bf 125}, Birkh\"auser Verlag (2001).

\bibitem{SH} J. Swift, P. Hohenberg. {\em Hydrodynamic fluctuations at the convective instability}. Phys. Rev. A {\bf 15}, no. 1, 319--328 (1977).

\bibitem{TGM} M. Tlidi, M. Georgiou, P. Mandel. {\em Transverse patterns in nascent optical bistability}. Phys. Rev. A {\bf 48}, no. 6, 4605--4609 (1993).

\bibitem{Ue} H. Uecker. {\em Diffusive stability of rolls in the two-dimensional real and complex Swift-Hohenberg equation}. Comm. Partial Differential Equations {\bf 24}, no. 11-12, 2109--2146 (1999).

\bibitem{W} L. Weis. {\em Operator-valued Fourier multiplier theorems and maximal $L_{p}$-regularity}. Math. Ann. {\bf 319}, no. 4, 735--758 (2001).

\end{thebibliography}
\end{document}